\documentclass[a4paper,twopage,reqno,11pt]{amsart}

\usepackage{pdfsync}
\usepackage{graphicx}
\usepackage{amsmath}
\usepackage{amssymb}
\usepackage{amsfonts}
\usepackage{amsthm}
\usepackage{amstext}
\usepackage{amsbsy}
\usepackage{amsopn}
\usepackage{amscd}
\usepackage{enumerate}
\usepackage{xcolor}
\usepackage{color}
\usepackage{url}
\usepackage[colorlinks]{hyperref}
\usepackage[numbers,sort&compress]{natbib}
\usepackage{tabularx}
\usepackage{multirow}
\usepackage{lipsum}
\usepackage{rotating}
\usepackage{longtable}
\usepackage[]{caption}
\usepackage{float}
\usepackage{lscape}
\usepackage{pdflscape}
\usepackage[]{algorithm2e}
\usepackage{color}
\usepackage[colorlinks]{hyperref}
\hypersetup{
    colorlinks=true,%
    citecolor=blue,%
    filecolor=black,%
    linkcolor=red,%
    urlcolor=magenta
}

\newtheorem{theorem}{Theorem}[section]

\newtheorem{lemma}[theorem]{Lemma}

\theoremstyle{definition}

\numberwithin{equation}{section}

\newcommand{\G}{\mathrm{G}}

\newcommand{\GL}{\mathrm{GL}}
\newcommand{\SL}{\mathrm{SL}}
\renewcommand{\L}{\mathrm{L}}
\newcommand{\Sp}{\mathrm{Sp}}

\newcommand{\SU}{\mathrm{SU}}

\newcommand{\Sz}{\mathrm{Sz}}
\renewcommand{\O}{\mathrm{O}}

\newcommand{\PSL}{\mathrm{PSL}}
\newcommand{\PSU}{\mathrm{PSU}}
\newcommand{\U}{\mathrm{U}}

\newcommand{\PSp}{\mathrm{PSp}}

\newcommand{\J}{\mathrm{J}}

\newcommand{\A}{\mathrm{A}}

\newcommand{\E}{\mathrm{E}}

\renewcommand{\S}{\mathrm{S}}

\newcommand{\Q}{\mathrm{Q}}
\newcommand{\QD}{\mathrm{QD}}

\newcommand{\D}{\mathrm{D}}

\newcommand{\nr}{\mathrm{nr}}

\newcommand{\He}{\mathrm{He}}
\newcommand{\McL}{\mathrm{McL}}
\newcommand{\Co}{\mathrm{Co}}
\newcommand{\HN}{\mathrm{HN}}
\newcommand{\ON}{\mathrm{O'N}}
\newcommand{\Suz}{\mathrm{Suz}}
\newcommand{\Fi}{\mathrm{Fi}}
\newcommand{\HS}{\mathrm{HS}}
\newcommand{\Ly}{\mathrm{Ly}}
\newcommand{\Th}{\mathrm{Th}}
\newcommand{\Ru}{\mathrm{Ru}}
\newcommand{\F}{\mathrm{F}}
\newcommand{\Z}{\mathrm{Z}}

\newcommand{\Aut}{\mathrm{Aut}}

\newcommand{\PG}{\mathrm{PG}}

\newcommand{\B}{\mathrm{B}}

\newcommand{\M}{\mathrm{M}}
\newcommand{\Fbb}{\mathbb{F}}

\newcommand{\Zbb}{\mathbb{Z}}

\newcommand{\Dmc}{\mathcal{D}}
\newcommand{\Bmc}{\mathcal{B}}

\newcommand{\Pmc}{\mathcal{P}}

\newcommand{\Smc}{\mathcal{S}}

\newcommand{\red}{\textcolor{red}}
\newcommand{\blue}{\textcolor{blue}}

\newcommand{\cyan}{\textcolor{cyan}}

\newcommand{\magenta}{\textcolor{magenta}}

\newcommand{\purple}{\textcolor{purple}}

\newcommand{\nsubK}{\magenta{\textsf{nsubK}}}
\newcommand{\nsubG}{\purple{\textsf{nsubG}}}
\newcommand{\nsubN}{\blue{\textsf{nsubN}}}
\newcommand{\norb}{\cyan{\textsf{norb}}}
\newcommand{\ndes}{\red{\textsf{ndes}}}

\renewcommand{\cos}{\textsf{Cos}}

\newcommand{\e}{\epsilon}
\renewcommand{\emptyset}{\varnothing}

\renewcommand{\leq}{\leqslant}
\renewcommand{\geq}{\geqslant}

\newcommand{\imod}[1]{\allowbreak\mkern4mu({\operator@font mod}\,\,#1)}

\DeclareMathOperator{\Fix}{Fix}

\begin{document}
 \title[]{On flag-transitive  automorphism groups of $2$-designs with $\lambda$ prime}

 \author[S.H. Alavi]{Seyed Hassan Alavi}
 \address{Seyed Hassan Alavi, Department of Mathematics, Faculty of Science, Bu-Ali Sina University, Hamedan, Iran.}
 \email{alavi.s.hassan@basu.ac.ir and  alavi.s.hassan@gmail.com}
 \author[A. Daneshkhah]{Ashraf Daneshkhah}
\address{Ashraf Daneshkhah, Department of Mathematics, Faculty of Science, Bu-Ali Sina University, Hamedan, Iran.}
\thanks{Corresponding author: Ashraf Daneshkhah}
\email{adanesh@basu.ac.ir and  daneshkhah.ashraf@gmail.com}
\author[A. Montinaro]{Alessandro Montinaro}%
\address{Alessandro Montinaro, Dipartimento di Matematica e Fisica ``E. De Giorgi'', University of Salento, Lecce, Italy. }%
\email{alessandro.montinaro@unisalento.it}

 \date{\today}%

\begin{abstract}
In this article, we study $2$-$(v,k,\lambda)$ designs $\Dmc$ with $\lambda$ prime admitting flag-transitive and point-primitive almost simple automorphism groups $G$ with socle $T$ a finite exceptional simple group or a sporadic simple groups. If the socle of $G$ is a finite exceptional simple group, then we prove that $\Dmc$ is isomorphic to one of two infinite families of $2$-designs with point-primitive automorphism groups,  one is the Suzuki-Tits ovoid design with parameter set $(v,b,r,k,\lambda)=(q^{2}+1,q^{2}(q^{2}+1)/(q-1),q^{2},q,q-1)$ design, where $q-1$ is a Mersenne prime, and the other is newly constructed in this paper and has parameter set $(v,b,r,k,\lambda)=(q^{3}(q^{3}-1)/2,(q+1)(q^{6}-1),(q+1)(q^{3}+1),q^{3}/2,q+1)$, where $q+1$ a Fermat prime.
If $T$ is a sporadic simple group, then we show that $\Dmc$ is isomorphic to a unique design admitting a point-primitive automorphism group with parameter set $(v,b,r,k,\lambda)=(176,1100,50,2)$, $(12,22,11,6,5)$ or $(22,77,21,6,5)$.
\end{abstract}

\maketitle

\section{Introduction}\label{sec:intro}

The main aim of this paper is to study $2$-designs with flag-transitive automorphism groups. In particular, we are interested in studying flag-transitive automorphism group of $2$-design with $\lambda$ prime. The symmetric $(v,k,2)$ designs (biplanes) have been investigated by Regueiro \cite{a:Regueiro-reduction}, and she classified all flag-transitive biplanes excluding $1$-dimensional affine automorphism groups \cite{a:Regueiro-alt-spor,a:Regueiro-reduction,a:Regueiro-classical,a:Regueiro-Exp}. Dong, Fang and Zhou studied flag-transitive automorphism groups of symmetric $(v,k,3)$ (triplanes), and in conclusion, they determined all such triplanes excluding $1$-dimensional affine automorphism groups  \cite{a:Zhou-lam3-affine,a:Zhou-lam3-spor,a:Zhou-lam3-alt,a:Zhou-lam3-excep,a:Zhou-lam3-classical}. In \cite{a:Zhou-sym-lam-prime} Z. Zhang, Y. Zhang and Zhou  generalized O'Reilly Regueiro's result~\cite{a:Regueiro-reduction} to prime $\lambda$, and proved that a flag-transitive and point-primitive automorphism group of a symmetric design with $\lambda$ prime must be of almost simple or affine type. Then by analysing these two types of primitive automorphism groups, and in conclusion, a classification of flag-transitive and point-primitive automorphism group of symmetric designs with $\lambda$ prime has been presented, except for the $1$-dimensional affine automorphism groups \cite{a:ABD-PrimeLam,a:ABD-Exp,a:ADM-PrimeLam-An,a:ABDM-HA-lam-prime}. 

In general, for flag-transitive $2$-designs with $\lambda$ prime, the case where $\lambda=2$ has been investigated in a series of papers \cite{a:A-Exp-lam2,a:Zhou-lam2-nonsym,a:DLPB-2021-lam2,a:Zhou-lam2-nonsym-An,a:ABDT-AS-lam2,a:Montinaro-lam2,a:Montinaro-PSL2lam2}. The point-primitive case is reduced to almost simple or affine type automorphism groups \cite{a:Zhang-LamP}. Therefore, in order to achieve a possible classification of such designs, we need to study the almost simple type and affine type automorphism groups. In this direction, Zhang and Shen \cite{a:Zhang-LamP-Exp} studied almost simple automorphism groups with socle finite exceptional simple groups, \cite{a:ABD-Exp,a:A-Exp-CP,a:D-Ree,a:Zhou-Exp-CP}. In this paper, we first revisit their proofs, and then we improve their main result \cite[Theorem 1.1]{a:Zhang-LamP-Exp} by obtaining all such possible $2$-designs:     

\begin{theorem}\label{thm:exp}
    Let $\Dmc$ be a nontrivial $2$-$(v,k,\lambda)$ design with $\lambda$ prime admitting a flag-transitive and point-primitive automorphism group $G$ with socle $T$ a finite exceptional simple group. Then (up to isomorphism) one of the following holds
    \begin{enumerate}[\rm (a)]
        \item $T$ is $^{2}\B_{2}(q)$ with $q^{2a+1}\geq 8$, and $\Dmc$ is the Suzuki-Tits ovoid design with parameter set $(v,b,r,k,\lambda)=(q^{2}+1,q^{2}(q^{2}+1)/(q-1),q^{2},q,q-1)$ design, where $q-1$ is a Mersenne prime;
        \item $T$ is $\G_{2}(q)$ with $q\geq 4$ even, and $\Dmc$ is the $2$-design with parameter set $(v,b,r,k,\lambda)=(q^{3}(q^{3}-1)/2
        ,(q+1)(q^{6}-1),(q+1)(q^{3}+1),q^{3}/2,q+1)$, where $q+1$ a Fermat prime and it is identified with the coset geometry $\cos(T,H,K)$, where $H=\SU_{3}(q){:}\Zbb_{2}$ and $K=[q^{6}]{:}\Zbb_{q-1}$.
    \end{enumerate}
\end{theorem}

The $2$-designs in Theorem~\ref{thm:exp}(a) appear in the study of flag-transitive $2$-designs with $\gcd(r,\lambda)=1$, see \cite{a:A-Exp-CP,a:Zhou-Exp-CP}. An explicit construction of such a $2$-design is given in \cite{a:A-Sz} using the natural action of Suzuki groups on Suzuki-Tits ovoid in $\PG_{3}(q)$. The $2$-designs in Theorem~\ref{thm:exp}(b) is new and we construct this design in Section \ref{sec:proof-exp}.  

%

In the second part of this paper, we study flag-transitive and point-primitive $2$-designs with $\lambda$ prime admitting almost simple sporadic automorphism groups: 

\begin{theorem}\label{thm:spor}
    Let $\Dmc$ be a nontrivial $2$-(v,k,$\lambda$) design with $\lambda$ prime admitting a flag-transitive and point-primitive automorphism group with socle a sporadic simple group. Then $(\Dmc,G)$ is (up to isomorphism) as one of the rows in {\rm Table~\ref{tbl:main}} and each design is constructed in the reference given in the last column of the table.
\end{theorem}
\begin{table}
    \scriptsize
    \centering
    \caption{Sporadic simple groups and flag-transitive $2$-designs with $\lambda$ prime.}\label{tbl:main}
    \begin{tabular}{clllllllllll}
        \noalign{\smallskip}\hline\noalign{\smallskip}
        Line &
        $v$ &
        $b$ &
        $r$ &
        $k$ &
        $\lambda$ &
        $G$ &
        $G_{\alpha}$ &
        $G_{B}$ &
        Reference \\
        \noalign{\smallskip}\hline\noalign{\smallskip}
        $1$ & $12$ & $22$ & $11$ & $6$ & $5$ & 
        $\M_{11}$ &
        $\PSL_{2}(11)$ &
        $\A_{6}$ &
        \cite{a:Zhan-CP-nonsym-sprodic}\\
        $2$ & $22$ &  $77$ &  $21$ &  $6$ &  $5$ & 
        $\M_{22}$ & 
        $\PSU_{3}(4)$ &
        $2^{4}{{:}}\A_{6}$ &
        \cite{a:Zhan-CP-nonsym-sprodic}\\
        & $22$ &  $77$ &  $21$ &  $6$ &  $5$ & 
        $\M_{22}{{:}}2$ & 
        $\PSU_{3}(4){{:}}2$ &
        $2^{4}{{:}}\S_{6}$ &
        \cite{a:Zhan-CP-nonsym-sprodic}\\
        $3$ &
        $176$ & 
        $1100$ & 
        $50$ & 
        $8$ & 
        $2$ & 
        $\HS$ &
        $\PSU_{3}(5){{:}}2$ &
        $\S_{8}$ &
        \cite{a:Zhou-lam2-nonsym}
        \\
        \noalign{\smallskip}\hline\noalign{\smallskip}
        \multicolumn{10}{l}{Note: $G_{\alpha}$ is the  point-stabiliser of $\alpha$ and $G_{B}$ is the  block-stabiliser of $B$.}
    \end{tabular}
\end{table}

Recently, the third author \cite{a:Montinaro-PrimLam} studied flag-transitive and point-imprimitive automorphism groups of $2$-designs with $\lambda$ prime. Indeed, by Theorem 1.1 and more precisely Theorem 9.2 in \cite{a:Montinaro-PrimLam}, we can easily show that there exists no flag-transitive $2$-design with $\lambda$ prime admitting a point-imprimitive almost simple automorphism group with socle an exceptional or a sporadic simple group. 

\subsection{Definitions and notation}

All groups and incidence structures in this paper are finite.  We denote by $\Fbb_{q}$ the Galois field of size $q$. Symmetric and alternating groups on $n$ letters are denoted by $\S_{n}$ and $\A_{n}$, respectively. We write ``$\Zbb_n$'' for the cyclic group of order $n$. We use ``$[n]$'' or ``$n$'' to denote a group of order $n$. A finite simple group is (isomorphic to) a cyclic group of prime order, an alternating group $\A_{n}$ for $n\geq 5$, a simple group of Lie type or a sporadic simple group, see \cite{b:Atlas} or \cite[Tables~5.1A-C]{b:KL-90}.
We use the standard notation as in \cite{b:Atlas} for finite simple groups. A $2$-design $\Dmc$ with parameters $(v,k,\lambda)$ is a pair $(\Pmc,\Bmc)$ with a set $\Pmc$ of $v$ points and a set $\Bmc$ of blocks of size $b$ such that each block is a $k$-subset of $\Pmc$ and each pair of two distinct points is contained in $\lambda$ blocks. Each point of $\Dmc$ is contained in exactly $r=bk/v$ blocks which is called the \emph{replication number} of $\Dmc$. 
In this paper, we refer to $(v,b,r,k,\lambda)$ as the parameter set of $\Dmc$. The design $\Dmc$ is symmetric if $v=b$ (or $r=k$), and it is nontrivial when $2 < k < v-1$. 
An \emph{automorphism} of a $2$-design $\Dmc$ is a permutation on the points mapping blocks to blocks. The full automorphism group $\Aut(\Dmc)$ of $\Dmc$ is the group consisting of all automorphisms of $\Dmc$. A \emph{flag} of $\Dmc$ is a point-block pair $(\alpha, B)$ such that $\alpha \in B$. For $G\leq \Aut(\Dmc)$, $G$ is called \emph{flag-transitive} if $G$ acts transitively on the set of flags. The group $G$ is said to be \emph{point-primitive} if $G$ acts primitively on $\Pmc$. Further notation and definitions in both design theory and group theory are standard and can be found, for example, in \cite{b:Atlas,b:Beth-I,b:KL-90}. We use GAP \cite{GAP4} for computations. 

\section{Proof of Theorem \ref{thm:exp}}\label{sec:proof-exp}

In this section, we prove Theorem \ref{thm:exp}. If $T\neq \G_{2}(q)$, then by the arguments given in \cite{a:Zhang-LamP-Exp}, we obtain part (a) of Theorem \ref{thm:exp}. However, when $T=\G_{2}(q)$, we need to revisit the proof of Lemma~3.8 in \cite{a:Zhang-LamP-Exp}. The arguments in the proof for the case where $t=1$ and $\e=-$ when $q$ is even needs further consideration as $\GL_2(q)$ with $q$ even contains cyclic groups of order $q^{2}-1$. We then improve the result by completely determining the designs arose in this case. 

\begin{lemma}\label{lem:3.8}
    If $T = G_2(q)$ and the type of $G_\alpha$ is $\SL_{3}^{\e}(q):2$ with $\e=\pm$, then $q$ is even and $\e=-$, $T$ is a flag-transitive automorphism group on $\Dmc$, and the parameter set of $\Dmc$ is $(v,b,r,k,\lambda)=(q^{3}(q^{3}-1),(q+1)(q^{6}-1),(q+1)(q^{3} +1),q^{3},q+1)$, where $\lambda=q+1$ is a Fermat prime.
\end{lemma}
\begin{proof}
    Recall from the proof of \cite[Lemma 3.8]{a:Zhang-LamP-Exp} that $t$ is a divisor of $q^3-\e1$, and 
    \[
    v=\frac{q^3(q^3)+\e1}{2}, 
    b=\frac{\lambda q^3 (q^6 - 1)}{2kt},
    r = \frac{\lambda(q^3 - \epsilon_1)}{2t}
    \text{ and } k = \frac{t(q^3 + \epsilon)}{2} + 1.
    \]
    As noted above, we only need to consider the case where $t=1$ and $\e=-$ when $q=2^f$ is even. Moreover, the point-stabiliser $T_{\alpha}$ is $\SU_{3}(q){:}\Zbb_{2}$, and 
    \begin{align}\label{eq:exp}
        | T_{B}| =\frac{f_{1}q^{6}(q^{2}-1)}{\lambda },
    \end{align}
    where $B$ is a block containing $\alpha$ and $f_1|T:T_B|=b$ for some divisor $f_1$ of $f$. In addition, 
    \[
    v=\frac{q^3(q^3-1)}{2}, \ 
    b=\lambda(q^{6}-1), \ 
    k = \frac{q^3}{2} 
    \text{ and }
    r = \lambda(q^{3}+1).
    \]
    In this case, $T_{B}$ is contained in a parabolic subgroup $M=R:\GL_{2}(q)$, where $R=[q^{5}]$. 
    We already know by \cite{a:A-Exp-lam2} that $\lambda$ is an odd prime. 
    
    It follows from \eqref{eq:exp} that $\lambda \mid f_{1}(q^{2}-1)$ since $q$ is even.
    As $T_{B}$ is contained in $M=R:\GL_{2}(q)$ of $\G_2(q)$, we conclude that $T_{B}$ contains a
    Sylow $2$-subgroup of $M$. Then $T_{B}/R$ may be viewed as a subgroup of $\GL_{2}(q)$ of order $f_{1}q(q^{2}-1)/\lambda $. If $\lambda \neq q+1$,  then $\SL_{2}(q)\leq
    T_{B}/R$, and hence $T_{B}=R:\SL_{2}(q)$ and $\lambda =f_{1}$. Thus $2q^{2}$ divides $|T_{\alpha ,B}\cap R|$ and $\SL_{2}(q)\leq
    T_{\alpha ,B}$ since $| T_{\alpha }:T_{\alpha ,B}|$ is a divisor of $ k=q^{3}/2$. However, $T_{\alpha}=\SU_{3}(q){:}\Zbb_{2}$ does not
    contain such a subgroup by \cite[Table 8.5]{b:BHR-Max-Low}. Therefore, $\lambda =q+1$ is a Fermat prime, and
    hence we obtain the parameter set of $\Dmc$ as in the statement of lemma. 
    Furthermore, we have $$(q+1)(q^{3}+1)=r=\left\vert G_{\alpha}:G_{\alpha,B}\right\vert=\left\vert G_{\alpha}/T_{\alpha}:T_{\alpha}G_{\alpha,B}/T_{\alpha}\right\vert \cdot \left\vert T_{\alpha}:T_{\alpha,B}\right\vert$$ since $G$ is flag-transitive on $\mathcal{D}$ and $T_{\alpha}$ is a normal subgroup of $G_{\alpha}$. Then $T_{\alpha}$ contains a Sylow $2$-subgroup of $T_{\alpha}$ since $r$ is odd, and so $q^{3}+1$ divides $\left\vert T_{\alpha}:T_{\alpha, B}\right \vert$ by \cite[Table 8.5]{b:BHR-Max-Low}. Note that $\lambda=q+1$ does not divide $\log_{2}q$, and hence  $\lambda$ does not divide the order of $\rm{Out}(T)$. Thus $q+1$ does not divide the order of $G_{\alpha}/T_{\alpha}$ since $G_{\alpha}/T_{\alpha}$ is isomorphic to a subgroup of $\rm{Out}(T)$. Therefore $\left\vert G_{\alpha}/T_{\alpha}:T_{\alpha}G_{\alpha,B}/T_{\alpha}\right\vert=1$ and $\left\vert T_{\alpha}:T_{\alpha,B}\right\vert=r=(q+1)(q^{3}+1)$ since $q+1$ is a prime and $q^{3}+1$ divides $\left\vert T_{\alpha}:T_{\alpha, B}\right \vert$, and hence $T$ acts flag-transitively on $\Dmc$.
\end{proof}

We now describe the $2$-design with this parameter set. To this end, we consider the group $\Sp_{6}(q)$ with $q$ even in its natural representation on $V_{6}(q)$. Let $\mathbf{f}$ be its invariant symplectic form. Recall that 
$\Sp_{6}(q)$ with $q$ even has a unique conjugacy class of subgroups isomorphic
to $\G_{2}(q)$ by \cite[Table 8.29]{b:BHR-Max-Low}. Let $T$ be any representative of
such a class of subgroups. The group $T$ has a unique conjugacy class of subgroups isomorphic
to $\SU_{3}(q){:}\Zbb_{2}$ by \cite[Table 8.30]{b:BHR-Max-Low}, and let $H$ be any
representative of such a class of subgroups. Then $H$ preserves an extension
field structure of $V_{6}(q)$ by \cite[Tables 8.28 and 8.30]{b:BHR-Max-Low}, that is to say, 
$H$ preserves $V_{3}(q^{2})$. Let $\mathbf{h}$ be the $H^{\prime }$-invariant
hermitian form. Since the non-degenerate symplectic form is unique up to
isometry and the non-degenerate Hermitian form of an odd dimensional
is invariant up to isometry, we may assume $\Sp_{6}(q)$, and  $H$ and  $T$ are taken in such a way that  $\mathbf{f}=\mathbf{h+h}^{q}$.

Let us focus on the action of $\PSp_{6}(q)$ on $\PG_{5}(q)$. The set $\Smc$ of all $1$-subspaces of $V_{3}(q^{2})$ regarded over $\Fbb_{q}$ is an $H$-invariant Desarguesian line spread of $\PG_{5}(q)$.
Moreover, the actions of $H$ on $\Smc$ and on the points of $\PG_{2}(q^{2})$ are equivalent, in particular, $H$ induces $\PSU_{3}(q):\Zbb_{2}$ on $\Smc$, and $\Smc$ consists of two $H$-orbits of length $%
q^{3}+1$ and $q^{2}(q^{2}-q+1)$ which correspond to a Hermitian unital of
order $q$ of $\PG_{2}(q^{2})$ and its complementary set. Denote by $\mathcal{H}$ the $\bar{H}$-orbit on $\Smc$ of length $q^{3}+1$, and let $Q=(P:C):\Zbb_{2}$ as a subgroup of   $\bar{H}$, where $P$ is a Sylow $2$-subgroup of $\bar{H}^{\prime }$ and $C=\Zbb_{q-1}$. We now consider the $\bar{H}$-action on $\PG_{2}(q^{2})$ (for instance, see \cite[Satz II.10.12]{b:Hupp-I}), we deduce that the following facts about $Q$ and some of its subgroups:

\begin{enumerate}
    \item $P:C$ fixes a unique line $\ell$ of $\Smc$ pointwise,
    and $\ell$ lies in $\mathcal{H}$. Further, $\ell$ is the unique $P$-invariant
    line of $\Smc$, and the $\Fix(P)=\ell $, where $\Fix(P)$ is the set of
    points of $PG_{5}(q)$ fixed by all elements of $P$. Moreover, $Q$ preserves $\ell$ and fixes a unique point on $\ell$, say $\alpha$.
    
    \item $(\alpha,\ell )$ is the unique $Q$-invariant incident point-line pair of $\PG_{5}(q)$. Indeed, let $\ell^{\prime}$ be any line of $\PG_{5}(q)$ preserved by $Q$. If $\ell^{\prime} \notin \mathcal{S}$, then $Q$ preserves the subset of $\mathcal{S}$ consisting of the $q+1$ lines intersecting $\ell^{\prime}$ in a point since $\mathcal{S}$ is $Q$-invariant. However, this is impossible since the $P$-orbits on $\mathcal{S}$ are one of length $1$, one of length $q^{2}$ and $q$ ones each of length $q^{3}$. Thus $\ell^{\prime} \in \mathcal{S}$, and hence $\ell^{\prime}=\ell$ by (1), and the assertion follows.    
\end{enumerate}

\begin{lemma}\label{Unique}
    There is a unique subgroup $K$ of $T$ of order $q^{6}(q-1)$ containing $Q$. Further, $K\cap H=Q$.
\end{lemma}
\begin{proof}
    Let $(\alpha,\ell )$ be the incident point-line pair of $\PG_{5}(q)$ preserved by $Q$ by (2). Hence, $Q\leq T_{\alpha,\ell}$. By \cite[Lemmas 5.1 and 5.2]{a:Cooperstein-G2-even}, $T_{\alpha}=R:L$, where $R$ is a $2$-group of order $q^{5}$ and $L\cong \GL_{2}(q)$, and $T_{\alpha,\ell}=R:(\Z(L)\times
    \F_{q(q-1)})$, where $\F_{q(q-1)}$ is a Frobenius group of order $q(q-1)$. Let $S$ be the (unique) Sylow $2$%
    -subgroup of $T_{\alpha, \ell}$, then $K:=S:C$ is a subgroup of $T_{\alpha, \ell}$ containing $Q$. Further, $K$ has order $q^{6}(q-1)$ and $K\cap H=Q$.
    
    Let $M$ be any further subgroup of $G$ of order $q^{6}(q-1)$ containing $Q$.
    Then $M=W:C$ with $W$ a Sylow $2$-subgroup of $G$ by \cite[Table 8.30]{b:BHR-Max-Low}. Moreover, $W$ fixes a unique incident point-line pair of $\PG_{5}(q)$ by \cite[Lemmas 5.1 and 5.2]{a:Cooperstein-G2-even}, say $(\alpha^{\prime},\ell^{\prime})$. Thus $C$ fixes $(\alpha^{\prime},\ell^\prime)$ since $C$ normalizes $W$, and hence $M$ fixes $(\alpha^{\prime},\ell^\prime)$ since $M=W:C$. In particular, $Q$ fixes $(\alpha^{\prime},\ell^\prime)$ since $Q<M$, and therefore $(\alpha^{\prime},\ell^\prime)=(\alpha,\ell)$ by (2). Thus $M \leq T_{\alpha,\ell}$. Since $T_{\alpha,\ell}$ has a unique Sylow $2$-subgroup as well as $Q \leq K \cap M$ and $\left\vert M\right\vert=\left\vert K\right\vert=q^{6}(q-1)$, we conclude that $M=K$.  
    
    
    
\end{proof}

For the group $T$ and its subgroups $H$ and $K$, we now consider the coset geometry $\Dmc_{0}:=\cos(T,H,K)$ which is the incidence structure whose point set and line (block) set are $\Pmc_{0}=\{ Hx :x \in T\}$ and $\Bmc_{0}=\{ Ky :y \in T\}$, respectively, and the incidence is given by nontrivial intersection, that is to say, $Hx$ is incident with $Ky$ if and only if $Hx\cap Ky\neq \emptyset$.

\begin{lemma}\label{lem:exp1}
    If $q\geq 4$ is even, then $\Dmc_{0}$ is a $2$-design with parameter set 
    $(v_{0},b_{0},r_{0},k_{0},\lambda_{0})= 
    (q^{3}(q^{3}-1)/2,
    (q+1)(q^{6}-1),
    (q+1)(q^{3}+1),
    q^{3}/2,
    q+1)$ design admitting $T$
    as a flag-transitive automorphism group.
\end{lemma}

\begin{proof}
    By \cite[Lemmas 1 and 2%
    ]{a:HigMcL-61}, the incidence structure $\Dmc_{0}$ has parameter set
    $(v_{0},b_{0},r_{0},k_{0})= 
    (q^{3}(q^{3}-1)/2,
    (q+1)(q^{6}-1),
    (q+1)(q^{3}+1),
    q^{3}/2)$
    and $G$ is a flag-transitive automorphism group of $\Dmc_{0}$. 
    By \cite[Table 8.30]{b:BHR-Max-Low}, the group $G$ has a unique conjugacy class of maximal subgroups isomorphic to $H$, and so the $G$-actions
    on $\mathcal{P}_{0}$ and on $\{ H^{x} : x \in G\}$ are permutationally isomorphic. 
    By \cite[Proposition 1]{a:LPS2}, the $H$-orbits on $\mathcal{P}_{0}\setminus \left\{ H\right\} $
    are $\mathcal{O}_{j}$, $j=1,...,q/2$, with $| \mathcal{O}%
    _{j}| =q^{2}(q^{3}+1)$, for $j=1,...,q/2-1$ and $| 
    \mathcal{O}_{q/2}| =(q^{2}-1)(q^{3}+1)$. Let $\mathcal{B}_{0}(H)$ be the subset of blocks of $\Dmc_{0}$
    incident with $H$. 
    Clearly, $\mathcal{B}_{0}(H)=\{ Ky : y \in
    H\}$, which is a $H$-orbit. Then $\left( \mathcal{O}_{j},\mathcal{B}%
    _{0}(H)\right) $, with $j=1,...,q/2$, is a $1$-design with parameter set $%
    (v_{j},b_{j},r_{j},k_{j})$  such that $b_{j}k_{j}=v_{j}r_{j}$, $j=1,...,q/2$%
    , by \cite[1.2.6]{b:Dembowski}. Note that $b_{j}=r_{0}=(q+1)(q^{3}+1)$ for each $%
    j=1,...,q/2$. Then 
    \begin{align}
        (q+1)(q^{3}+1)k_{j}=\left\{ 
        \begin{array}{ll}
            q^{2}(q^{3}+1)r_{j}, & j=1,...,q/2-1; \\ 
            (q^{2}-1)(q^{3}+1)r_{q/2}, & j=q/2.%
        \end{array}%
        \right.  \label{rel}
    \end{align}%
    Further, 
    \begin{equation}
        \sum_{j=1}^{q/2}k_{j}=k-1=q^{3}/2-1 \mathit{.}  \label{rel2}
    \end{equation}%
    It follows from (\ref{rel}) that $k_{j}=q^{2}k_{j}^{\prime }$ with $%
    k_{j}^{\prime }\geq 1$ for $j=1,...,q/2-1$ and $q-1\mid k_{q/2}$. Then by \eqref{rel2}, we have that  
    \[
    \sum_{j=1}^{q/2-1}q^{2}k_{j}^{\prime }+(k_{q/2}+1)=q^{3}/2, 
    \]
    and so $k_{q/2}+1=\mu q^{2}$ for some $\mu \geq 1$. Therefore, we get%
    \[
    \sum_{j=1}^{q/2-1}k_{j}^{\prime }+\mu =q/2 
    \]
    with $k_{j}^{\prime }\geq 1$ for $j=1,...,q/2-1$, and $\mu \geq 1$. Hence, $%
    k_{j}^{\prime }=\mu =1$ for $j=1,...,q/2-1$. Thus $k_{j}=q^{2}$ for $%
    j=1,...,q/2-1$ and $k_{q/2}=q^{2}-1$. Therefore, $r_{j}=q+1$ for each $%
    j=1,...,q/2$, and hence in $\Dmc_{0}$, we observe that each pair of points is incidence with $q+1$ blocks, that is to say, $\lambda_{0}=q+1$.  This completes the proof.
\end{proof}

\begin{lemma}\label{lem:exp2}	
    Let $\Dmc$ be a nontrivial $2$-$(v,k,\lambda)$ design with $\lambda$ prime admitting a flag-transitive and point-primitive automorphism group $G$ with socle $T=\G_{2}(q)$ with $q\geq 4$ even. If the parameter set of $\Dmc$ is  $(v,b,r,k,\lambda)=(q^{3}(q^{3}-1)/2,(q+1)(q^{6}-1),(q+1)(q^{3}+1),q^{3}/2,q+1)$, then $\Dmc$ is isomorphic to $\Dmc_{0}$.
\end{lemma}

\begin{proof}
    Let $(\alpha,B)$ be a flag of $\Dmc=(\Pmc,\Bmc)$. Then $T_{\alpha}=\SU_{3}(q):\Zbb_{2}$, $T_{B}=[q^{6}]:\Zbb_{q-1}$ and 
    $T_{\alpha,B}\cong ([q^{3}]:\Zbb_{q-1}):\Zbb_{2}$ by \cite[Lemmas 1 and 2%
    ]{a:HigMcL-61} since $T$ acts flag-transitively on $\mathcal{D}$ by Lemma \ref{lem:3.8}. By \cite[Table 8.30]{b:BHR-Max-Low}, since $G$ has a unique conjugacy
    class of subgroups isomorphic to $\SU_{3}(q):\Zbb_{2}$ acting transitively on $\mathcal{P}$, we may assume that $T_{\alpha}=H$. Moreover, $H$ has a unique conjugacy class of subgroups isomorphic to $([q^{3}]:\Zbb_{q-1}):\Zbb_{2}$ by \cite[Satz II.10.12]{b:Hupp-I}. Thus we may assume $T_{\alpha,B}=Q$. Then $T_{B}=K$ by Lemma \ref{Unique}. Therefore, \cite[Lemma 1]{a:HigMcL-61} implies that $\Dmc$ is isomorphic to $\Dmc_{0}$.
\end{proof}

\begin{proof}[\rm \bf Proof of Theorem \ref{thm:exp}]
    Let $G$ be a flag-transitive almost simple automorphism group of $\Dmc$ with socle a finite exceptional simple group $T$. Suppose that $G$ is point-primitive. If $T\neq \G_{2}(q)$, then by the arguments given in \cite{a:Zhang-LamP-Exp}, we obtain part (a) of Theorem \ref{thm:exp} follows. This design is (isomorphic to) the Suzuki-Tits ovoid design explicitly constructed in \cite{a:A-Sz} using the natural action of Suzuki group on Suzuki-Tits ovoid in $\PG_{3}(q)$, see also \cite{a:A-Exp-CP,a:Zhou-Exp-CP}. If $T= \G_{2}(q)$, then in Lemma \ref{lem:3.8}, we revisited the proof of Lemma~3.8 in \cite{a:Zhang-LamP-Exp}, and in conclusion we obtain the parameter set given in part (b) of the theorem and that $T$ acts flag-transitively on $\mathcal{D}$. The existence and uniqueness of the design with this parameter set is followed by Lemmas \ref{lem:exp1} and \ref{lem:exp2}. In fact, the design in this case is (isomorphic to) the coset geometry $\cos(T,H,K)$, where $H=\SU_{3}(q):\Zbb_{2}$ and $K=[q^6]:\Zbb_{q-1}$.
    
    In order to complete the proof, we need to show that any $T \unlhd G \leq \rm{Aut}(T)$ is isomorphic to a group of automorphisms of $\cos(T,H,K)$. Firstly, note that $G=T:\left\langle \phi \right\rangle $ where $\phi $ is a $2$-element of order a divisor of $\log_{2}q$ since $\lambda =q+1$ is a Fermat prime. Since $T$ contains a unique conjugate class of subgroups isomorphic to $\SU_{3}(q):\mathbb{Z}_{2}$, we may assume that $\left\langle \phi \right\rangle $ normalizes a copy of $\SU_{3}(q):\mathbb{Z}_{2}$, say $H_{0}$. In particular, $\left\langle \phi \right\rangle $ normalizes one of its $q^{3}+1$ maximal subgroups isomorphic to $([q^{3}]:\mathbb{Z}_{q-1}):\mathbb{Z}_{2}$, say $Q_{0}$, and hence $%
    \left\langle \phi \right\rangle $ normalizes the unique subgroup of $T$
    isomorphic to $[q^{6}]:\mathbb{Z}_{q-1}$ containing $Q_{0}$ by Lemma \ref%
    {Unique} applied to $Q_{0}$. Clearly, there is $\gamma \in T$ such that $H^{\gamma }=H_{0}$,
    and there is $\delta \in H_{0}$ such that $Q^{\gamma \delta }=Q_{0}$. 
    Set $\vartheta :=\gamma \delta $. Then $H^{\vartheta }=H_{0}$ and $Q^{\vartheta
    }=Q_{0}$, and hence $K^{\vartheta }=K_{0}$ again by Lemma \ref{Unique}. Now, we define $\bar{\vartheta}:\cos(T,H,K) \rightarrow \cos(T,H_{0},K_{0})$ given by 
    $(Hx,Ky)^{\bar{\vartheta}}=\left( H_{0}x^{\vartheta},K_{0}y^{\vartheta }\right)$. Then $\bar{\vartheta}$
    is a bijection such that $Hx\cap Ky\neq \varnothing $ if
    and only if $H_{0}x^{\vartheta }\cap K_{0}y^{\vartheta }=\left( Hx\cap
    Ky\right) ^{\vartheta }\neq \varnothing $. Thus $\bar{\vartheta}$ is an
    isomorphism from $\cos(T,H,K)$ onto $\cos(T,H_{0},K_{0})$. 
    
    Let $H_{0}x$ and $H_{0}y$ with $x,y\in T$ be two right cosets of $H_0$ and $K_{0}$ in $T$, respectively, such that $H_{0}\left\langle \phi \right\rangle x=H_{0}\left\langle \phi
    \right\rangle y$. Then $xy^{-1}\in H_{0}\left\langle \phi
    \right\rangle \cap T=H_{0}$, and hence $H_{0}x=H_{0}y$. Therefore, $H_{0}x=H_{0}y$ if and only if $H_{0}\left\langle \phi \right\rangle x=H_{0}\left\langle \phi \right\rangle y$ for $x,y\in T$. Further, 
    $\left\vert T:H_{0}\right\vert =\left\vert T\left\langle \phi \right\rangle
    :H_{0}\left\langle \phi \right\rangle \right\vert =q^{3}(q^{3}-1)/2$. Hence $\left\{ H_{0}\left\langle \phi \right\rangle x : x\in T\right\}$ is the set of all right cosets of $H_{0}\left\langle \phi \right\rangle$ in $G$. Similarly, $\left\{ K_{0}\left\langle
    \phi \right\rangle y : y\in T\right\} $ is the set of all right cosets of $K_{0}\left\langle \phi \right\rangle $ in $G$. 
    Finally, if $H_{0}x\cap
    K_{0}y\neq \varnothing $ with $x,y \in T$, then $xy^{-1}\in H_{0}K_{0} \subseteq H_{0}K_{0}\left\langle \phi \right\rangle=H_{0}\left\langle \phi \right\rangle K_{0}\left\langle \phi \right\rangle$ since $\left\langle \phi \right\rangle$ normalizes $K_{0}$, and hence $H_{0}\left\langle \phi
    \right\rangle x\cap K_{0}\left\langle \phi \right\rangle y\neq \varnothing $ since also $H_{0}$ is normalized by $\left\langle \phi \right\rangle$. Conversely, $H_{0}\left\langle \phi
    \right\rangle x\cap K_{0}\left\langle \phi \right\rangle y\neq \varnothing $ with $x,y \in T$ implies $xy^{-1}\in H_{0}K_{0}\left\langle \phi \right\rangle \cap T=H_{0}K_{0}$ and hence $H_{0}x\cap K_{0}y\neq \varnothing $. Therefore the map  $\bar{\mu}:\cos(T,H_{0},K_{0}) \rightarrow \cos(G,H_{0}\left\langle \phi \right\rangle ,K_{0}\left\langle \phi \right\rangle)$ given by $(H_{0}x,K_{0}y)\mapsto \left( H_{0}\left\langle\phi \right\rangle x,K_{0}\left\langle \phi \right\rangle y\right)$ is an
    isomorphism from $\cos(T,H_{0},K_{0})$ onto $\cos(G,H_{0}\left\langle \phi \right\rangle ,K_{0}\left\langle \phi \right\rangle)$, and hence $\bar{\vartheta}\bar{\mu}$
    is an isomorphism from $\cos(T,H_{0},K_{0})$ onto $\cos(G,H_{0}\left\langle \phi \right\rangle ,K_{0}\left\langle \phi \right\rangle)$. Further, $G$ is a group of automorphisms (right translations) of $\cos(G,H_{0}\left\langle \phi \right\rangle ,K_{0}\left\langle \phi \right\rangle)$. Thus $G$ is isomorphic to a group of automorphisms of $\cos(T,H,K)$, and hence the proof is completed.
\end{proof}

\section{Sporadic simple groups}\label{proof:spor}

In this section, we prove Theorem \ref{thm:spor}. Here, we use GAP \cite{GAP4} for computations, in particular, we use stored groups library and software packages ``\verb|AtlasRep|'' and ``\verb|Design|''.

\begin{proof}[\rm \bf Proof of Theorem \ref{thm:spor}]
    
    Suppose that $\Dmc = (\Pmc, \Bmc)$ is a nontrivial $2$-$(v,k,\lambda)$ design admitting a flag-transitive and point-primitive automorphism group $G$. Suppose also that $G$ is an almost simple group whose socle is a sporadic simple group. 
    We obtain no example for symmetric designs by \cite[Theorem 1.1]{a:AD-spor-ft}. 
    Thus we always assume that $\Dmc$ is non-symmetric, that is to say, $v<b$ or $k<r$ (by Fisher's inequality \cite{b:Beth-I}). 
    Since $G$ is point-primitive, any point-stabiliser of $G$ is maximal in $G$ \cite[Corollary~1.5A]{b:Dixon}.
    If $\lambda=2$, then we obtain one unique design in line 1 of Table~\ref{tbl:main} with parameter set  $(v,b,r,k,\lambda)=(176, 1100, 50, 8, 2 )$ when $G=\HS$. In addition, the case where $\lambda$ does not divide $r$, that is to say, $\gcd(r,\lambda)=1$, has been treated in \cite{a:Zhan-CP-nonsym-sprodic}, and we obtain the designs on lines $2$-$3$. Therefore, we focus on the case where $\lambda\geq 3$ is a prime divisor of the parameter $r$. 
    
    If  $(\alpha,B)$ is a flag of $\Dmc$, then $H:=G_{\alpha}$ is maximal in $G$ and  the block-stabiliser $K:=G_{B}$ is contained in some maximal subgroup $N$ of $G$. In what follows, we frequently use the following facts: 
    \begin{enumerate}[\rm \quad (i)]
        \item $r(k-1)=\lambda(v-1)$;
        \item $rv=bk$;
        \item $\lambda v<r^2$;
        \item $v=|G:H|$, $b=|G:K|$, $r=|H:L|$ and $k=|K:L|$, where $L:=G_{\alpha,B}=H_{B}=K_{\alpha}$;
        \item $r\mid |H|$ and $|G|<|H|^{3}$;
        \item $r\mid \lambda e$, for all nontrivial subdegrees $e$ of $G$;
        \item $K$ is transitive on $B$, and so $B$ is a $K$-orbit.  
    \end{enumerate}
    We note here that the properties (iv)-(vii) follow from the fact that $G$ is flag-transitive, see for example \cite{a:Davies-87} and \cite[Lemma~2.1 and Corollary~2.1]{a:ABCD-PrimeRep}.  We now follow the steps below and we prove that the case where $\lambda\geq 3$ is a prime divisor of $r$ leads to no possible parameter set. \smallskip
    
    \noindent \textbf{(1)} Since $H$ is a maximal subgroup of $G$ and all maximal subgroups of sporadic simple groups are known by \cite{b:Atlas,a:DLP-Monster}, we know all possibilities for $H$. Since also $H$ is a large subgroup of $G$, that is to say $|G|\leq |H|^{3}$, the subgroup $H$ is known for each group $G$, see \cite[Theorem 3 and Proposition~6.2]{a:AB-Large-15}. Indeed, all maximal subgroups of the groups listed as in \eqref{eq:1} are large, and the list of large maximal subgroups of the remaining groups $G$ are recorded in Table~\ref{tbl:large}. 
    \begin{align}\label{eq:1}
        \nonumber	&\M_{11}, \M_{12}, \M_{12}{:}2, 	\M_{22}, \M_{22}{:}2, \M_{23}, \J_{2}{:}2,  \J_{3}, \\ 
        &\HS,	\HS{:}2,	\McL, \McL{:}2, \He{:}2, \Fi_{22}, \Fi_{22}{:}2, 	\HN{:}2. 
    \end{align}
    
    \begin{table}
        \scriptsize
        \caption{Large maximal subgroups of some almost sporadic  simple groups.}\label{tbl:large}
        \begin{tabular}{lp{15cm}}
            \noalign{\smallskip}\hline\noalign{\smallskip}
            Group & Subgroups \\
            \noalign{\smallskip}\hline\noalign{\smallskip}
            %
            %
            %
            $\M_{24}$ &  $\M_{23}$, $\M_{22}.2$, $2^4:\A_{8}$, $\M_{12}.2$, $2^6:3.\S_{6}$, $\L_3(4).3.2_2$, 
            $2^6:(\L_{3}(2){\times}\S_{3})$, $\L_{2}(23)$  \\ 
            $\J_1$ &  $\L_2(11)$, $2^3.7.3$, $2{\times}\A_5$, $19:6$, $11:10$, $\D_6{\times}\D_{10}$  \\ 
            $\J_{2}$ &  $\U_3(3)$, $3.\A_6.2_2$, $2^{1+4}:\A_5$, $2^{2+4}.3{\times}\S_3$, $\A_4{\times}\A_5$, 
            $\A_5{\times}\D_{10}$, $\L_3(2).2$, $5^2:\D_{12}$  \\ 
            $\J_3:2$ &  $\J_3$, $\L_{2}(16).4$, $2^4:(3{\times}\A_{5}).2$, $\L_{2}(17){\times}2$, $(3{\times}\M_{10}):2$, 
            $3^2.3^{1+2}:8.2$, $2^{1+4}.\S_{5}$, $2^{2+4}:(\S_{3}{\times}\S_{3})$  \\ 
            $\He$ &  $\S_{4}(4).2$, $2^2.\L_{3}(4).\S_{3}$, $2^6:3.\S_{6}$, $2^6:3.\S_{6}$, $2^1+6.\L_{3}(2)$, 
            $7^2:2\L_{2}(7)$, $3.\A_{7}.2$, $7^{1+2}:(\S_{3}{\times}3)$, $\S_{4}{\times}\L_{3}(2)$, $7:3{\times}\L_{3}(2)$ 
            \\ 
            $\Ru$ &  $^{2}\F_{4}(2)'.2$, $2^6:\U_{3}(3):2$, $(2^2{\times}\Sz(8)):3$, $2^3+8:\L_3(2)$, $\U_3(5).2$, 
            $2.2^4+6:\S_{5}$, $\L_{2}(25).2^2$, $\A_{8}$, $\L_{2}(29)$, $5^2:4\S_{5}$  \\ 
            $\Suz$ &  $\G_{2}(4)$, $3_2.\U_4(3).2_3'$, $\U_{5}(2)$, $2^{1+6}.\U_{4}(2)$, $3^5:\M_{11}$, $\J_{2}.2$, 
            $2^{4+6}:3\A_{6}$, $(\A_{4}{\times}\L_{3}(4)):2$, $2^2+8(\A_{5}{\times}\S_{3})$, $\M_{12}.2$, 
            $3^{2+4}:2(2^2{\times}\A_{4})2$, $(\A_{6}{\times}\A_{5}).2$, $(3^2:4{\times}\A_{6}).2$, $\L_{3}(3).2$, 
            $\L_{3}(3).2$, $\L_{2}(25)$  \\ 
            $\Suz.2$ &  $\Suz$, $\G_2(4).2$, $3_2.\U_4(3).(2^2)_{133}$, $\U_{5}(2).2$, $2^{1+6}_-.\U_4(2).2$, 
            $3^5:(\M_{11}{\times}2)$, $\J_2.2{\times}2$, $2^{4+6}:3\S_{6}$, $(\A_{4}{\times}\L_3(4):2_3):2$, 
            $2^{2+8}:(\S_{5}{\times}\S_{3})$, $\M_{12}.2{\times}2$, $3^{2+4}:2(\S_{4}{\times}\D_8)$, 
            $(\A_{6}:2_2{\times}\A_{5}).2$, $(3^2:8{\times}\A_{6}).2$, $\L_{2}(25).2_2$  \\ 
            $\ON$ &  $\L_{3}(7).2$, $\L_{3}(7).2$, $\J_{1}$, $4_2.\L_3(4).2_1$, $(3^3:4{\times}\A_{6}).2$, 
            $3^4:2^{1+4}\D_{10}$, $\L_{2}(31)$, $\L_{2}(31)$, $4^3.\L_3(2)$, $\M_{11}$, $\M_{11}$  \\ 
            $\Co_3$ &  $\McL.2$, $\HS$, $\U_4(3).(2^2)_{133}$, $\M_{23}$, $3^5:(2{\times}\M_{11})$, $2.\S_{6}(2)$, 
            $\U_3(5).3.2$, $3^{1+4}:4\S_{6}$, $2^4.\A_{8}$, $\L_3(4).\D_{12}$, $2{\times}\M_{12}$, $2^2.(2^7.3^2).\S_{3}$, 
            $\S_{3}{\times}\L_{2}(8).3$  \\ 
            $\ON.2$ &  $\ON$, $\J_1{\times}2$, $4_2.\L_3(4).2.2$, $(3^2:4{\times}\A_{6}).2^2$, $3^4:2^{1+4}.(5:4)$, 
            $4^3.(\L_3(2){\times}2)$, $7^{1+2}_{+}:(3{\times}\D_{16})$  \\ 
            $\Co_2$ &  $\U_6(2).2$, $2^{10}:\M_{22}:2$, $\McL$, $2^{1+8}:\S_{6}(2)$, $\HS.2$, $(2^{4}{\times}2^{1+6}).\A_{8}$, 
            $\U_4(3).\D_8$, $2^{4+10}(\S_{5}{\times}\S_{3})$, $\M_{23}$, $3^{1+4}:2^{1+4}.\S_{5}$  \\ 
            $\HN$ &  $\A_{12}$, $2.\HS.2$, $\U_3(8).3_1$, $2^{1+8}.(\A_{5}{\times}\A_{5}).2$, $(\D_{10}{\times}\U_3(5)).2$, 
            $5^{1+4}:2^{1+4}.5.4$, $2^6.\U_4(2)$, $(\A_{6}{\times}\A_{6}).\D_8$, $2^3.2^2.2^6.(3{\times}\L_3(2))$, 
            $5^2.5.5^2.4\A_{5}$, $\M_{12}.2$, $\M_{12}.2$, $3^4:2(\A_{4}{\times}\A_{4}).4$  \\ 
            $\Ly$ &  $\G_2(5)$, $3.\McL.2$, $5^3.\L_{3}(5)$, $2.A_{11}$, $5^{1+4}:4\S_{6}$, $3^5:(2 {\times} M_{11})$, 
            $3^{2+4}:2\A_{5}.\D_8$  \\ 
            $\Th$ &  $^{3}\D_4(2).3$, $2^5.L_5(2)$, $2^{1+8}.\A_9$, $\U_3(8).6$, $(3{\times}\G_2(3)):2$, 
            $3.3^2.3.(3 \times 3^2).3^2:2\S_4$, $3^2.3^3.3^2.3^2:2\S_4$  \\ 
            $\Fi_{23}$ &  $2.\Fi_{22}$, $\O_8^+(3).3.2$, $2^2.\U_6(2).2$, $\S_{8}(2)$, $\S_{3}{\times}\O_7(3)$, 
            $2^{11}.\M_{23}$, $3^{1+8}.2^{1+6}.3^{1+2}.2\S_{4}$, $3^{10}.(\L_3(3) {\times} 2)$, $\S_{12}$, 
            $(2^2{\times}2^{1+8}).(3{\times}\U_4(2)).2$, $2^{6+8}:(\A_{7}{\times}\S_{3})$, $\S_{4}{\times}\S_{6}(2)$, 
            $\S_{4}(4).4$  \\ 
            $\Co_1$ &  $\Co_2$, $3.\Suz.2$, $2^{11}:\M_{24}$, $\Co_3$, $2^{1+8}_{+}.\O_8^+(2)$, $\U_6(2).3.2$, 
            $(\A_{4}{\times}\G_2(4)):2$, $2^{2+12}:(\A_{8}{\times}\S_{3})$, $2^{4+12}.(\S_{3}{\times}3\S_{6})$, 
            $3^2.\U_4(3).\D_8$, $3^6:2\M_{12}$, $(\A_{5}{\times}\J_2):2$, $3^{1+4}.2\U_4(2).2$, 
            $(\A_{6}{\times}\U_{3}(3)):2$, $3^{3+4}:2(\S_{4}{\times}\S_{4})$  \\ 
            $\J_4$ &  $2^{11}:\M_{24}$, $2^{1+12}.3.\M_{22}:2$, $2^{10}:\L_5(2)$, $2^{3+12}.(\S_5 {\times} \L_3(2))$, 
            $\U_3(11).2$  \\ 
            $\Fi_{24}'$ &  $\Fi_{23}$, $2.\Fi_{22}.2$, $(3{\times}\O_{8}^{+}(3):3):2$, $\O_{10}^{-}(2)$, $3^7.\O_7(3)$, 
            $3^{1+10}:\U_{5}(2):2$, $2^{11}:\M_{24}$, $2^2.\U_6(2).\S_{3}$, $2^{1+12}.3_1.\U_4(3).2_2'$, 
            $3^3.3^{10}.\GL_{3}(3)$, $3^2.3^4.3^8.(\A_{5}{\times}2\A_{4}).2$, $(\A_{4}{\times}\O_8^+(2).3).2$, $\He.2$, 
            $\He.2$, $2^{3+12}.(\L_3(2){\times}\A_{6})$, $2^{6+8}.(\S_{3}{\times}\A_{8})$, $(3^2:2{\times}\G_2(3)).2$  \\ 
            $\Fi_{24}':2$ &  $\Fi_{24}'$, $2{\times}\Fi_{23}$, $2^2.Fi_{22}.2$, $\S_{3}{\times}\O_8^{+}(3):\S_{3}$, 
            $\O_{10}^{-}(2).2$, $3^7.\O_7(3):2$, $3^{1+10}:(2{\times}\U_{5}(2):2)$, $2^{12}.\M_{24}$, 
            $2^2.\U_6(2):\S_{3}{\times}2$, $2^{1+12}_{+}.3_1.\U_4(3).2^2_{122}$, $3^3.3^{10}.(\L_{3}(3){\times}2^2)$, 
            $3^2.3^4.3^8.(\S_{5}{\times}2\S_{4})$, $\S_{4}{\times}\O_8^{+}(2):\S_{3}$, $2^{3+12}.(\L_3(2){\times}\S_{6})$, 
            $2^{7+8}.(\S_{3}{\times}\A_{8})$, $(\S_{3}{\times}\S_{3}{\times}\G_2(3)):2$  \\ 
            $\B$ &  $2.^2\E_6(2).2$, $2^{1+22}.\Co_2$, $\Fi_{23}$, $2^{9+16}.\S_{8}(2)$, $\Th$, $(2^2{\times}\F_4(2)):2$, 
            $2^{2+10+20}.(\M_{22}.2{\times}\S_{3})$, $2^{30}.\L_5(2)$, $\S_{3}{\times}\Fi_{22}.2$, 
            $2^{35}.(\S_{5}{\times}\L_3(2))$, $\HN.2$, $\O_8^{+}(3).\S_{4}$  \\ 
            $\M$ &  $2.\B$, $2^{1+24}.\Co_1$, $3.\Fi_{24}$, $2^2.^2\E_6(2):\S_{3}$, $2^{10+16}.\O_{10}^{+}(2)$, 
            $2^{2+11+22}.(\M_{24}{\times}\S_{3})$, $3^{1+12}.2\Suz.2$, $2^{5+10+20}.(\S_{3}{\times}\L_5(2))$  \\ 
            %
            \noalign{\smallskip}\hline\noalign{\smallskip}
        \end{tabular}  
    \end{table}
    
    \noindent \textbf{(2)} For each group $G$, since $\lambda$ is an odd prime divisor of $|G|$, we easily find the possibilities of $\lambda$. Fixing a subgroup $H$ obtained in (1), we use (iv), and so  $v=|G:H|$. For each prime $\lambda$, by (i) and (v), the parameter $r$ is a divisor of $\gcd(\lambda(v-1), |H|)$, and then we can find $k=1+\lambda(v-1)/r$ and $b=rv/k$ by taking into account of the conditions $v<b$ and $\lambda\mid r$ 
    Therefore, for each pair $(G,H)$, we obtain all possible parameter sets. For each group, the number of such candidate parameter sets is recorded in Table \ref{tbl:stats} below, and in total, there are $124$ possibilities $[G,H,(v,b,r,k,\lambda)]$ recorded in Table~\ref{tbl:poss} for further consideration. In Table~\ref{tbl:poss}, the number $\nr(H)$ indicates the number associated to the maximal subgroup $H$ of $G$ in the library of software package ``\verb|AtlasRep|'' in GAP \cite{GAP4}.\smallskip
    
    \begin{table}
        \scriptsize
        \caption{The number of parameter sets for sporadic almost simple groups.}\label{tbl:stats}
        \begin{tabularx}{\textwidth}{lc|lc|lc|lc|lc}
            \noalign{\smallskip}\hline\noalign{\smallskip}
            Group & \# parameters &
            Group & \# parameters &
            Group & \# parameters &
            Group & \# parameters &
            Group & \# parameters \\
            \noalign{\smallskip}\hline\noalign{\smallskip}
            $\M_{11}$ & $4$ & 
            $\J_2$ & $3$ & 
            $\Suz : 2$ & $0$ & 
            $\ON : 2$ & $0$ &
            $\Fi_{24}' : 2$ & $0$ \\ 
            $\M_{12}$ & $0$ & 
            $\J_{2} : 2$ & $3$ & 
            $\McL$ & $6$ &  
            $\Co_{1}$ & $0$ & 
            $\HN$ & $0$ \\ 
            $\M_{12} : 2$ & $0$ & 
            $\J_{3}$ & $0$ & 
            $\McL : 2$ & $0$ & 
            $\Co_{2}$ & $0$ & 
            $\HN : 2$ & $0$ \\ 
            $\M_{22}$ & $5$ & 
            $\J_{3} : 2$ & $0$ &
            $\Ru$ & $0$ & 
            $\Co_3$ & $9$ 
            & $\Th$ & $0$ \\
            $\M_{22} : 2$ & $3$ & 
            $\J_{4}$ & $0$ & 
            $\He$ & $0$ & 
            $\Fi_{22}$ & $0$ & 
            $\B$ & $0$ \\
            $\M_{23}$ & $43$ & 
            $\HS$ & $19$ & 
            $\He : 2$ & $0$ & 
            $\Fi_{22} : 2$ & $0$ 
            & $\M$ & $0$ \\
            $\M_{24}$ & $8$ & 
            $\HS : 2$ & $9$ & 
            $\Ly$ & $0$ & 
            $\Fi_{23}$ & $0$ & \\
            $\J_{1}$ & $6$ & 
            $\Suz$ & $0$ & 
            $\ON$ & $6$ & 
            $\Fi_{24}'$ & $0$ & \\
            \noalign{\smallskip}\hline\noalign{\smallskip}
        \end{tabularx} 
    \end{table}
  
    \noindent \textbf{(3)} Since $K$ is contained in a maximal subgroup $N$ of $G$, it follows that $b=|G:K|=|G:N|\cdot|N:K|$, and so the index of a maximal subgroup of $G$ must divide $b$. Then for each tuple $[G,H,(v,b,r,k,\lambda)]$, this gives a list of candidates for maximal subgroups $N$. All these possibilities are recorded in Tables~\ref{tbl:poss-subdeg} and \ref{tbl:poss-final}. Then we find the subdegrees of such a group $G$ in its right coset action on the set of right cosets of $H$. Since the degree $v$ of the permutation representation of $G$ in this action is at most $2025$, we can use the library of primitive permutation groups in GAP \cite{GAP4}. For each $[G,H,(v,b,r,k,\lambda)]$ as in one of the rows of Tables~\ref{tbl:poss-subdeg}, we present the subdegrees in which the notation $m^{n}$ means that the multiplicity of the subdegree $m$ is $n$. However, all the possibilities recorded in Tables~\ref{tbl:poss-subdeg} can be ruled out as for each row, (the smallest nontrivial) subdegree in the last column of the table violates the property (vi) which says that $r$ is a divisor of $\lambda e$, for all nontrivial subdegrees $e$. This leaves the possibilities in Table~\ref{tbl:poss-final} for further inspection. For convenience, for each row in this table, in the third and fifth columns, we record the numbers $\nr(H)$ and $\nr(N)$  which  indicate the numbers associated to the maximal subgroups  $H$ and $N$ of $G$, respectively,  in the library of software package ``\verb|AtlasRep|'' in GAP \cite{GAP4}. \smallskip
    
    We now analyse each possibility listed in Table~\ref{tbl:poss-final}. These are $166$ candidates in total. For $[G,H,(v,b,r,k,\lambda)]$, we first consider the group $G$ as a primitive permutation group on $v$ points.
    we know that $G$ must have a subgroup $K$ of index $b$, and if there exist such a subgroup $K$ of $G$, then it needs to be isomorphic to a subgroup of $N$ and it should contain a subgroup $L$ of index $k$, see (iv). By inspecting the subgroups of $G$ using GAP, in the last column of the table, we used the notation ``\nsubG'' to indicate that there exist no subgroups $K$ of $G$ of index $b$. Otherwise, we give the structure of the subgroup $K$ in the sixth column of the table. In the case where $G$ has a subgroup of index $b$, we have two scenarios, either $K$ is not isomorphic to a subgroup of $N$ which we write ``$\nsubN$'', or $K$ may be viewed as a subgroup of $N$ but its index is not $k$ which we record ``$\nsubK$'' in the last column of Table~\ref{tbl:poss-final}.
    
    For the remaining possibilities,  we examine the $K$-orbits with $K$ as a permutation group on $v$ points. This is to see if these orbits give rise to a possible $2$-design. Indeed, since $G$ is flag-transitive, the block set $\Bmc$ must be a $G$-orbit $B^{G}$, where the base block $B$ is a $K$-orbit on $v$ points. This fact rules out the possibilities in Table~\ref{tbl:poss-final} with the notation ``\norb'' recorded in their last column. In conclusion, we obtain five possibilities in which the $K$-orbits are 
    \begin{align*}
        B_1&=\{50, 64, 101, 142, 187, 202, 242\} \text{ when $G=\M_{23}$;}\\
        B_2&=\{4, 38, 63, 134, 162, 200, 215\} \text{ when $G=\M_{23}$;}\\
        B_3&=\{ 5, 34, 69, 149, 201, 221, 243\} \text{ when $G=\M_{23}$;}\\
        B_4&=\{6, 12, 23, 24, 26, 34, 39, 42, 68, 70, 78, 86\} \text{ when $G=\J_{2}$;}\\
        B_5&=\{10, 30, 36, 44, 46, 49, 64, 67, 68, 75, 86, 93\}\text{ when $G=\J_{2}:2$}.
    \end{align*}
    However, by using the software package ``\verb|design|'' in GAP \cite{GAP4}, we conclude that none of these $K$-orbits forms a base block with the associated parameter set $(v,b,r,k,\lambda)$, and so we use the notation ``\ndes'' in the last column of Table~\ref{tbl:poss-final} to address this observation. Therefore, we only obtain the $2$-designs listed in Table~\ref{tbl:main}.       
\end{proof}

\clearpage

\scriptsize

\begin{longtable}{lll>{\raggedright\arraybackslash}p{12cm}} 
    \caption{Possible parameters $(v,b,r,k,\lambda)$ for pairs $(G,H)$.}\label{tbl:poss}  \\ 
    \hline 
    $G$ & $H$ & $\nr(H)$ & $(v,b,r,k,\lambda)$ \\
    \hline 
    \endfirsthead 
    \multicolumn{4}{c}{{\normalsize \scshape \tablename\ \thetable{}} - continued} \\ \\ 
    \hline 
    $G$ & $H$ & $\nr(H)$ & $(v,b,r,k,\lambda)$ \\
    \hline 
    \endhead 
    \hline 
    \multicolumn{4}{r}{{Continued}}\endfoot
    \endlastfoot
    $\M_{11}$ & $\A_6.2_3$ & $1$ & $( 11, 55, 15, 3, 3 )$ \\ 
    $\M_{11}$ & $3^2{:}\Q_{8}.2$ & $3$ & $( 55, 99, 18, 10, 3 )$, $(  55, 165, 30, 10, 5 )$, $(  55, 363, 66, 10, 11 )$ \\ 
    $\J_1$ & $19{:}6$ & $4$ & $( 1540, 15675, 285, 28, 5 )$, $(  1540, 21945, 399, 28, 7 )$, $(  1540, 34485, 627, 28, 11 )$ \\ 
    $\J_1$ & $11{:}10$ & $5$ & $( 1596, 8778, 165, 30, 3 )$, $(  1596, 20482, 385, 30, 7 )$, $(  1596, 55594, 1045, 30, 19 )$ \\ 
    $\M_{22}$ & $\A_7$ & $3$ & $( 176, 3080, 105, 6, 3 )$ \\ 
    $\M_{22}$ & $\A_7$ & $4$ & $( 176, 3080, 105, 6, 3 )$ \\ 
    $\M_{22}$ & $\A_6.2_3$ & $7$ & $( 616, 660, 45, 42, 3 )$, $(  616, 1540, 105, 42, 7 )$, $(  616, 2420, 165, 42, 11 )$ \\ 
    $\J_2$ & $\PSU_3(3)$ & $1$ & $( 100, 225, 27, 12, 3 )$, $(  100, 375, 45, 12, 5 )$, $(  100, 525, 63, 12, 7 )$ \\ 
    $\M_{22}{:}2$ & $\A_6.2^2$ & $6$ & $( 616, 660, 45, 42, 3 )$, $(  616, 1540, 105, 42, 7 )$, $(  616, 2420, 165, 42, 11 )$ \\ 
    $\J_2{:}2$ & $\PSU_3(3).2$ & $2$ & $( 100, 225, 27, 12, 3 )$, $(  100, 375, 45, 12, 5 )$, $(  100, 525, 63, 12, 7 )$ \\ 
    $\M_{23}$ & $\M_{22}$ & $1$ & $( 23, 253, 33, 3, 3 )$ \\ 
    $\M_{23}$ & $\PSL_3(4).2_2$ & $2$ & $( 253, 414, 36, 22, 3 )$, $(  253, 4554, 126, 7, 3 )$, $(  253, 15939, 252, 4, 3 )$, $(  253, 690, 60, 22, 5 )$, $(  253, 1518, 90, 15, 5 )$, $(  253, 3542, 140, 10, 5 )$, $(  253, 7590, 210, 7, 5 )$, $(  253, 15939, 315, 5, 5 )$, $(  253, 26565, 420, 4, 5 )$, $(  253, 53130, 630, 3, 5 )$, $(  253, 966, 84, 22, 7 )$, $(  253, 1518, 132, 22, 11 )$, $(  253, 16698, 462, 7, 11 )$, $(  253, 58443, 924, 4, 11 )$, $(  253, 116886, 1386, 3, 11 )$, $(  253, 3174, 276, 22, 23 )$, $(  253, 34914, 966, 7, 23 )$, $(  253, 122199, 1932, 4, 23 )$, $(  253, 244398, 2898, 3, 23 )$ \\ 
    $\M_{23}$ & $2^4{:}\A_7$ & $3$ & $( 253, 414, 36, 22, 3 )$, $(  253, 4554, 126, 7, 3 )$, $(  253, 15939, 252, 4, 3 )$, $(  253, 690, 60, 22, 5 )$, $(  253, 1518, 90, 15, 5 )$, $(  253, 3542, 140, 10, 5 )$, $(  253, 7590, 210, 7, 5 )$, $(  253, 15939, 315, 5, 5 )$, $(  253, 26565, 420, 4, 5 )$, $(  253, 53130, 630, 3, 5 )$, $(  253, 966, 84, 22, 7 )$, $(  253, 1518, 132, 22, 11 )$, $(  253, 16698, 462, 7, 11 )$, $(  253, 58443, 924, 4, 11 )$, $(  253, 116886, 1386, 3, 11 )$, $(  253, 3174, 276, 22, 23 )$, $(  253, 34914, 966, 7, 23 )$, $(  253, 122199, 1932, 4, 23 )$, $(  253, 244398, 2898, 3, 23 )$ \\ 
    $\M_{23}$ & $\M_{11}$ & $5$ & $( 1288, 5313, 165, 40, 5 )$, $(  1288, 45540, 495, 14, 5 )$, $(  1288, 63756, 693, 14, 7 )$, $(  1288, 209484, 2277, 14, 23 )$ \\ 
    $\HS$ & $\M_{22}$ & $1$ & $( 100, 330, 33, 10, 3 )$, $(  100, 2475, 99, 4, 3 )$, $(  100, 375, 45, 12, 5 )$, $(  100, 550, 55, 10, 5 )$, $(  100, 4125, 165, 4, 5 )$, $(  100, 525, 63, 12, 7 )$, $(  100, 770, 77, 10, 7 )$, $(  100, 5775, 231, 4, 7 )$, $(  100, 825, 99, 12, 11 )$ \\ 
    $\HS$ & $\PSU_3(5).2$ & $2$ & $( 176, 1650, 75, 8, 3 )$, $(  176, 3080, 105, 6, 3 )$, $(  176, 2750, 125, 8, 5 )$, $(  176, 3850, 175, 8, 7 )$, $(  176, 6050, 275, 8, 11 )$ \\ 
    $\HS$ & $\PSU_3(5).2$ & $3$ & $( 176, 1650, 75, 8, 3 )$, $(  176, 3080, 105, 6, 3 )$, $(  176, 2750, 125, 8, 5 )$, $(  176, 3850, 175, 8, 7 )$, $(  176, 6050, 275, 8, 11 )$ \\ 
    $\HS{:}2$ & $\M_{22}.2$ & $2$ & $( 100, 330, 33, 10, 3 )$, $(  100, 2475, 99, 4, 3 )$, $(  100, 375, 45, 12, 5 )$, $(  100, 550, 55, 10, 5 )$, $(  100, 4125, 165, 4, 5 )$, $(  100, 525, 63, 12, 7 )$, $(  100, 770, 77, 10, 7 )$, $(  100, 5775, 231, 4, 7 )$, $(  100, 825, 99, 12, 11 )$ \\ 
    $\M_{24}$ & $\M_{22}.2$ & $2$ & $( 276, 7590, 165, 6, 3 )$, $(  276, 17710, 385, 6, 7 )$, $(  276, 58190, 1265, 6, 23 )$ \\ 
    $\M_{24}$ & $\M_{12}.2$ & $4$ & $( 1288, 27324, 297, 14, 3 )$, $(  1288, 5313, 165, 40, 5 )$, $(  1288, 45540, 495, 14, 5 )$, $(  1288, 63756, 693, 14, 7 )$, $(  1288, 209484, 2277, 14, 23 )$ \\ 
    $\McL$ & $\M_{22}$ & $2$ & $( 2025, 22275, 264, 24, 3 )$, $(  2025, 37125, 440, 24, 5 )$, $(  2025, 51975, 616, 24, 7 )$ \\ 
    $\McL$ & $\M_{22}$ & $3$ & $( 2025, 22275, 264, 24, 3 )$, $(  2025, 37125, 440, 24, 5 )$, $(  2025, 51975, 616, 24, 7 )$ \\ 
    $\ON$ & $\PSL_3(7).2$ & $1$ & $( 122760, 123690, 931, 924, 7 )$, $(  122760, 194370, 1463, 924, 11 )$, $(  122760, 547770, 4123, 924, 31 )$ \\ 
    $\ON$ & $\PSL_3(7).2$ & $2$ & $( 122760, 123690, 931, 924, 7 )$, $(  122760, 194370, 1463, 924, 11 )$, $(  122760, 547770, 4123, 924, 31 )$ \\ 
    $\Co_3$ & $\McL.2$ & $1$ & $( 276, 1725, 75, 12, 3 )$, $(  276, 7590, 165, 6, 3 )$, $(  276, 2875, 125, 12, 5 )$, $(  276, 12650, 275, 6, 5 )$, $(  276, 4025, 175, 12, 7 )$, $(  276, 17710, 385, 6, 7 )$, $(  276, 6325, 275, 12, 11 )$, $(  276, 13225, 575, 12, 23 )$, $(  276, 58190, 1265, 6, 23 )$ \\ 
    \hline
\end{longtable}

\begin{longtable}{lllllllllll} 
    \caption{ The parameters which are ruled out by a subdegree of $G$.}\label{tbl:poss-subdeg}  \\ 
    \hline 
    $G$ & $H$ & $\nr(H)$ & $v$ &  $b$ & $r$ & $k$ & $\lambda$ & Subdegrees &Comment \\
    \hline 
    \endfirsthead 
    \multicolumn{10}{c}{{\normalsize \scshape \tablename\ \thetable{}} - continued} \\ \\ 
    \hline 
    $G$ & $H$ & $\nr(H)$ & $v$ &  $b$ & $r$ & $k$ & $\lambda$ &Subdegrees &Comment \\
    \hline 
    \endhead 
    \hline 
    \multicolumn{10}{r}{{Continued}}\endfoot
    \endlastfoot
    $\M_{23}$ & $2^4{:}\A_{7}$ & $3$ & $253$ & $414$ & $36$ & $22$ & $3$  & $1^{1}112^{1}140^{1}$  & $112$ \\
    $\M_{23}$ & $2^4{:}\A_{7}$ & $3$ & $253$ & $690$ & $60$ & $22$ & $5$  & $1^{1}112^{1}140^{1}$  & $112$ \\
    $\M_{23}$ & $2^4{:}\A_{7}$ & $3$ & $253$ & $966$ & $84$ & $22$ & $7$  & $1^{1}112^{1}140^{1}$  & $112$ \\
    $\M_{23}$ & $2^4{:}\A_{7}$ & $3$ & $253$ & $1518$ & $90$ & $15$ & $5$  & $1^{1}112^{1}140^{1}$  & $112$ \\
    $\M_{23}$ & $2^4{:}\A_{7}$ & $3$ & $253$ & $1518$ & $132$ & $22$ & $11$  & $1^{1}112^{1}140^{1}$  & $112$ \\
    $\M_{23}$ & $2^4{:}\A_{7}$ & $3$ & $253$ & $3174$ & $276$ & $22$ & $23$  & $1^{1}112^{1}140^{1}$  & $112$ \\
    $\M_{23}$ & $2^4{:}\A_{7}$ & $3$ & $253$ & $4554$ & $126$ & $7$ & $3$  & $1^{1}112^{1}140^{1}$  & $112$ \\
    $\M_{23}$ & $2^4{:}\A_{7}$ & $3$ & $253$ & $7590$ & $210$ & $7$ & $5$  & $1^{1}112^{1}140^{1}$  & $112$ \\
    $\M_{23}$ & $2^4{:}\A_{7}$ & $3$ & $253$ & $15939$ & $252$ & $4$ & $3$  & $1^{1}112^{1}140^{1}$  & $112$ \\
    $\M_{23}$ & $2^4{:}\A_{7}$ & $3$ & $253$ & $15939$ & $315$ & $5$ & $5$  & $1^{1}112^{1}140^{1}$  & $112$ \\
    $\M_{23}$ & $2^4{:}\A_{7}$ & $3$ & $253$ & $16698$ & $462$ & $7$ & $11$  & $1^{1}112^{1}140^{1}$  & $112$ \\
    $\M_{23}$ & $2^4{:}\A_{7}$ & $3$ & $253$ & $26565$ & $420$ & $4$ & $5$  & $1^{1}112^{1}140^{1}$  & $112$ \\
    $\M_{23}$ & $2^4{:}\A_{7}$ & $3$ & $253$ & $34914$ & $966$ & $7$ & $23$  & $1^{1}112^{1}140^{1}$  & $112$ \\
    $\M_{23}$ & $2^4{:}\A_{7}$ & $3$ & $253$ & $53130$ & $630$ & $3$ & $5$  & $1^{1}112^{1}140^{1}$  & $112$ \\
    $\M_{23}$ & $2^4{:}\A_{7}$ & $3$ & $253$ & $58443$ & $924$ & $4$ & $11$  & $1^{1}112^{1}140^{1}$  & $112$ \\
    $\M_{23}$ & $2^4{:}\A_{7}$ & $3$ & $253$ & $116886$ & $1386$ & $3$ & $11$  & $1^{1}112^{1}140^{1}$  & $112$ \\
    $\M_{23}$ & $2^4{:}\A_{7}$ & $3$ & $253$ & $122199$ & $1932$ & $4$ & $23$  & $1^{1}112^{1}140^{1}$  & $112$ \\
    $\M_{23}$ & $2^4{:}\A_{7}$ & $3$ & $253$ & $244398$ & $2898$ & $3$ & $23$  & $1^{1}112^{1}140^{1}$  & $112$ \\
    $\M_{23}$ & $\PSL_3(4).2_2$ & $2$ & $253$ & $414$ & $36$ & $22$ & $3$  & $1^{1}42^{1}210^{1}$ & $42$ \\
    $\M_{23}$ & $\PSL_3(4).2_2$ & $2$ & $253$ & $690$ & $60$ & $22$ & $5$  & $1^{1}42^{1}210^{1}$ & $42$ \\
    $\M_{23}$ & $\PSL_3(4).2_2$ & $2$ & $253$ & $966$ & $84$ & $22$ & $7$  & $1^{1}42^{1}210^{1}$ & $42$ \\
    $\M_{23}$ & $\PSL_3(4).2_2$ & $2$ & $253$ & $1518$ & $90$ & $15$ & $5$  & $1^{1}42^{1}210^{1}$ & $42$ \\
    $\M_{23}$ & $\PSL_3(4).2_2$ & $2$ & $253$ & $1518$ & $132$ & $22$ & $11$  & $1^{1}42^{1}210^{1}$ & $42$ \\
    $\M_{23}$ & $\PSL_3(4).2_2$ & $2$ & $253$ & $3174$ & $276$ & $22$ & $23$  & $1^{1}42^{1}210^{1}$ & $42$ \\
    $\M_{23}$ & $\PSL_3(4).2_2$ & $2$ & $253$ & $3542$ & $140$ & $10$ & $5$  & $1^{1}42^{1}210^{1}$ & $42$ \\
    $\M_{23}$ & $\PSL_3(4).2_2$ & $2$ & $253$ & $15939$ & $252$ & $4$ & $3$  & $1^{1}42^{1}210^{1}$ & $42$ \\
    $\M_{23}$ & $\PSL_3(4).2_2$ & $2$ & $253$ & $15939$ & $315$ & $5$ & $5$  & $1^{1}42^{1}210^{1}$ & $42$ \\
    $\M_{23}$ & $\PSL_3(4).2_2$ & $2$ & $253$ & $26565$ & $420$ & $4$ & $5$  & $1^{1}42^{1}210^{1}$ & $42$ \\
    $\M_{23}$ & $\PSL_3(4).2_2$ & $2$ & $253$ & $53130$ & $630$ & $3$ & $5$  & $1^{1}42^{1}210^{1}$ & $42$ \\
    $\M_{23}$ & $\PSL_3(4).2_2$ & $2$ & $253$ & $58443$ & $924$ & $4$ & $11$  & $1^{1}42^{1}210^{1}$ & $42$ \\
    $\M_{23}$ & $\PSL_3(4).2_2$ & $2$ & $253$ & $116886$ & $1386$ & $3$ & $11$  & $1^{1}42^{1}210^{1}$ & $42$ \\
    $\M_{23}$ & $\PSL_3(4).2_2$ & $2$ & $253$ & $122199$ & $1932$ & $4$ & $23$  & $1^{1}42^{1}210^{1}$ & $42$ \\
    $\M_{23}$ & $\PSL_3(4).2_2$ & $2$ & $253$ & $244398$ & $2898$ & $3$ & $23$  & $1^{1}42^{1}210^{1}$ & $42$ \\
    $\M_{23}$ & $\M_{11}$ & $5$ & $1288$ & $45540$ & $495$ & $14$ & $5$  & $1^{1}165^{1}330^{1}792^{1}$  & $165$ \\
    $\M_{23}$ & $\M_{11}$ & $5$ & $1288$ & $63756$ & $693$ & $14$ & $7$  & $1^{1}165^{1}330^{1}792^{1}$  & $165$ \\
    $\M_{23}$ & $\M_{11}$ & $5$ & $1288$ & $209484$ & $2277$ & $14$ & $23$  & $1^{1}165^{1}330^{1}792^{1}$  & $165$ \\
    $\M_{24}$ & $\M_{22}.2$ & $2$ & $276$ & $7590$ & $165$ & $6$ & $3$ & $1^{1}44^{1}231^{1}$  & $44$ \\
    $\M_{24}$ & $\M_{22}.2$ & $2$ & $276$ & $17710$ & $385$ & $6$ & $7$ & $1^{1}44^{1}231^{1}$  & $44$ \\
    $\J_1$ & $19{:}6$ & $4$ & $1540$ & $15675$ & $285$ & $28$ & $5$ & $1^{1}19^{1}38^{4}57^{6}114^{9}$ & $19$ \\
    $\J_1$ & $19{:}6$ & $4$ & $1540$ & $21945$ & $399$ & $28$ & $7$ & $1^{1}19^{1}38^{4}57^{6}114^{9}$ & $19$ \\
    $\J_1$ & $19{:}6$ & $4$ & $1540$ & $34485$ & $627$ & $28$ & $11$ & $1^{1}19^{1}38^{4}57^{6}114^{9}$ & $19$ \\
    $\J_1$ & $11{:}10$ & $5$ & $1596$ & $8778$ & $165$ & $30$ & $3$  & $1^{1}11^{1}22^{2}55^{2}110^{13}$ & $11$ \\
    $\J_1$ & $11{:}10$ & $5$ & $1596$ & $20482$ & $385$ & $30$ & $7$  & $1^{1}11^{1}22^{2}55^{2}110^{13}$ & $11$ \\
    $\J_1$ & $11{:}10$ & $5$ & $1596$ & $55594$ & $1045$ & $30$ & $19$  & $1^{1}11^{1}22^{2}55^{2}110^{13}$ & $11$ \\
    $\HS$ & $\M_{22}$ & $1$ & $100$ & $4125$ & $165$ & $4$ & $5$ & $1^{1}22^{1}77^{1}$ & $22$ \\
    $\HS$ & $\M_{22}$ & $1$ & $100$ & $5775$ & $231$ & $4$ & $7$ & $1^{1}22^{1}77^{1}$ & $22$ \\
    $\HS{:}2$ & $\M_{22}.2$ & $2$ & $100$ & $4125$ & $165$ & $4$ & $5$ & $1^{1}22^{1}77^{1}$ & $22$ \\
    $\HS{:}2$ & $\M_{22}.2$ & $2$ & $100$ & $5775$ & $231$ & $4$ & $7$ & $1^{1}22^{1}77^{1}$ & $22$ \\
    $\McL$ & $\M_{22}$ & $2$ & $2025$ & $22275$ & $264$ & $24$ & $3$ & $1^{1}330^{1}462^{1}1232^{1}$ & $330$ \\
    $\McL$ & $\M_{22}$ & $2$ & $2025$ & $37125$ & $440$ & $24$ & $5$ & $1^{1}330^{1}462^{1}1232^{1}$ & $330$ \\
    $\McL$ & $\M_{22}$ & $2$ & $2025$ & $51975$ & $616$ & $24$ & $7$ & $1^{1}330^{1}462^{1}1232^{1}$ & $330$ \\
    $\McL$ & $\M_{22}$ & $3$ & $2025$ & $22275$ & $264$ & $24$ & $3$ & $1^{1}330^{1}462^{1}1232^{1}$ & $330$ \\
    $\McL$ & $\M_{22}$ & $3$ & $2025$ & $37125$ & $440$ & $24$ & $5$ & $1^{1}330^{1}462^{1}1232^{1}$ & $330$ \\
    $\McL$ & $\M_{22}$ & $3$ & $2025$ & $51975$ & $616$ & $24$ & $7$ & $1^{1}330^{1}462^{1}1232^{1}$ & $330$ \\
    \hline
\end{longtable}

\begin{longtable}{lllllllllllll} 
    \caption{Some possible parameters for some almost sporadic simple groups.} \label{tbl:poss-final}  \\ 
    \hline 
    $G$ & $H$ & $\nr(H)$ & $N$ & $\nr(N)$ & $K$ & $v$ &  $b$ & $r$ & $k$ & $\lambda$ & Orbit & Comment \\
    \hline 
    \endfirsthead 
    \multicolumn{13}{c}{{\normalsize \scshape \tablename\ \thetable{}} - continued} \\ \\ 
    \hline 
    $G$ & $H$ & $\nr(H)$ & $N$ & $\nr(N)$ & $K$ & $v$ &  $b$ & $r$ & $k$ & $\lambda$ & Orbit & Comment  \\
    \hline 
    \endhead 
    \hline 
    \multicolumn{13}{r}{{Continued}}\endfoot
    \endlastfoot
    $\M_{11}$ & $\A_6.2_3$ & $1$ & $\A_6.2_3$ & $1$ & $3^2 {:} \QD_{16}$ & $11$ & $55$ & $15$ & $3$ & $3$ & - & \nsubN \\
    $\M_{11}$ & $\A_6.2_3$ & $1$ & $3^2{:}\Q_8.2$ & $3$ & $3^2 {:} \QD_{16}$ & $11$ & $55$ & $15$ & $3$ & $3$ & - & \nsubK \\
    $\M_{11}$ & $3^2{:}\Q_8.2$ & $3$ & $\A_6.2_3$ & $1$ & - & $55$ & $99$ & $18$ & $10$ & $3$ & - & \nsubG \\
    $\M_{11}$ & $3^2{:}\Q_8.2$ & $3$ & $\A_6.2_3$ & $1$ & $\GL_2(3)$ & $55$ & $165$ & $30$ & $10$ & $5$ & - & \nsubN \\
    $\M_{11}$ & $3^2{:}\Q_8.2$ & $3$ & $3^2{:}\Q_8.2$ & $3$ & $\GL_2(3)$ & $55$ & $165$ & $30$ & $10$ & $5$ & - & \nsubN \\
    $\M_{11}$ & $3^2{:}\Q_8.2$ & $3$ & $2.\S_4$ & $5$ & $\GL_2(3)$ & $55$ & $165$ & $30$ & $10$ & $5$ & - & \nsubK \\
    $\M_{11}$ & $3^2{:}\Q_8.2$ & $3$ & $\A_6.2_3$ & $1$ & - & $55$ & $363$ & $66$ & $10$ & $11$ & - & \nsubG \\
    $\M_{22}$ & $\A_7$ & $3$ & $\PSL_3(4)$ & $1$ & $\A_4^2$ & $176$ & $3080$ & $105$ & $6$ & $3$ & - & \nsubN \\
    $\M_{22}$ & $\A_7$ & $3$ & $2^4{:}\A_6$ & $2$ & $\A_4^2$ & $176$ & $3080$ & $105$ & $6$ & $3$ & - & \norb \\
    $\M_{22}$ & $\A_7$ & $3$ & $\A_6.2_3$ & $7$ & $\A_4^2$ & $176$ & $3080$ & $105$ & $6$ & $3$ & - & \nsubN \\
    $\M_{22}$ & $\A_7$ & $4$ & $\PSL_3(4)$ & $1$ & $\A_4^2$ & $176$ & $3080$ & $105$ & $6$ & $3$ & - & \nsubN \\
    $\M_{22}$ & $\A_7$ & $4$ & $2^4{:}\A_6$ & $2$ & $\A_4^2$ & $176$ & $3080$ & $105$ & $6$ & $3$ & - & \norb \\
    $\M_{22}$ & $\A_7$ & $4$ & $\A_6.2_3$ & $7$ & $\A_4^2$ & $176$ & $3080$ & $105$ & $6$ & $3$ & - & \nsubN \\
    $\M_{22}$ & $\A_6.2_3$ & $7$ & $\PSL_3(4)$ & $1$ & - & $616$ & $660$ & $45$ & $42$ & $3$ & - & \nsubG \\
    $\M_{22}$ & $\A_6.2_3$ & $7$ & $2^3{:}\SL_3(2)$ & $6$ & - & $616$ & $660$ & $45$ & $42$ & $3$ & - & \nsubG \\
    $\M_{22}$ & $\A_6.2_3$ & $7$ & $\PSL_3(4)$ & $1$ & $\A_4^2 {:} 2$ & $616$ & $1540$ & $105$ & $42$ & $7$ & - & \nsubN \\
    $\M_{22}$ & $\A_6.2_3$ & $7$ & $2^4{:}\A_6$ & $2$ & $\A_4^2 {:} 2$ & $616$ & $1540$ & $105$ & $42$ & $7$ & - & \nsubK \\
    $\M_{22}$ & $\A_6.2_3$ & $7$ & $\PSL_3(4)$ & $1$ & - & $616$ & $2420$ & $165$ & $42$ & $11$ & - & \nsubG \\
    $\M_{22}{:}2$ & $\A_6.2^2$ & $6$ & $\M_{22}$ & $1$ & $2^3 {:} \PSL_3(2)$ & $616$ & $660$ & $45$ & $42$ & $3$ & - & \norb \\
    $\M_{22}{:}2$ & $\A_6.2^2$ & $6$ & $\PSL_3(4).2_2$ & $2$ & $2^3 {:} \PSL_3(2)$ & $616$ & $660$ & $45$ & $42$ & $3$ & - & \nsubN \\
    $\M_{22}{:}2$ & $\A_6.2^2$ & $6$ & $2{\times}2^3{:}\PSL_3(2)$ & $5$ & $2^3 {:} \PSL_3(2)$ & $616$ & $660$ & $45$ & $42$ & $3$ & - & \norb \\
    $\M_{22}{:}2$ & $\A_6.2^2$ & $6$ & $\M_{22}$ & $1$ & $\A_4^2 {:} 4$ & $616$ & $1540$ & $105$ & $42$ & $7$ & - & \nsubK \\
    $\M_{22}{:}2$ & $\A_6.2^2$ & $6$ & $\M_{22}$ & $1$ & $(((2^4 {:} 3) {:} 2) {:} 3) {:} 2$ & $616$ & $1540$ & $105$ & $42$ & $
    7$ & - & \nsubN \\
    $\M_{22}{:}2$ & $\A_6.2^2$ & $6$ & $\M_{22}$ & $1$ & $\S_4^2$ & $616$ & $1540$ & $105$ & $42$ & $7$ & - & \nsubN \\
    $\M_{22}{:}2$ & $\A_6.2^2$ & $6$ & $\PSL_3(4).2_2$ & $2$ & $\A_4^2 {:} 4$ & $616$ & $1540$ & $105$ & $42$ & $7$ & - & \nsubN \\
    $\M_{22}{:}2$ & $\A_6.2^2$ & $6$ & $\PSL_3(4).2_2$ & $2$ & $(((2^4 {:} 3) {:} 2) {:} 3) {:} 2$ & $616$ & $1540$ & $105$ & $
    42$ & $7$ & - & \nsubN \\
    $\M_{22}{:}2$ & $\A_6.2^2$ & $6$ & $\PSL_3(4).2_2$ & $2$ & $\S_4^2$ & $616$ & $1540$ & $105$ & $42$ & $7$ & - & \nsubN \\
    $\M_{22}{:}2$ & $\A_6.2^2$ & $6$ & $2^4{:}\S_6$ & $3$ & $\A_4^2 {:} 4$ & $616$ & $1540$ & $105$ & $42$ & $7$ & - & \nsubK \\
    $\M_{22}{:}2$ & $\A_6.2^2$ & $6$ & $2^4{:}\S_6$ & $3$ & $(((2^4 {:} 3) {:} 2) {:} 3) {:} 2$ & $616$ & $1540$ & $105$ & $42$ & $
    7$ & - & \nsubK \\
    $\M_{22}{:}2$ & $\A_6.2^2$ & $6$ & $2^4{:}\S_6$ & $3$ & $\S_4^2$ & $616$ & $1540$ & $105$ & $42$ & $7$ & - & \nsubK \\
    $\M_{22}{:}2$ & $\A_6.2^2$ & $6$ & $\M_{22}$ & $1$ & - & $616$ & $2420$ & $165$ & $42$ & $11$ & - & \nsubG \\
    $\M_{22}{:}2$ & $\A_6.2^2$ & $6$ & $\PSL_3(4).2_2$ & $2$ & - & $616$ & $2420$ & $165$ & $42$ & $11$ & - & \nsubG \\
    $\M_{23}$ & $\M_{22}$ & $1$ & $\M_{22}$ & $1$ & $2^4 {:} A7$ & $23$ & $253$ & $33$ & $3$ & $3$ & - & \nsubN \\
    $\M_{23}$ & $\M_{22}$ & $1$ & $\M_{22}$ & $1$ & $\PSL_3(4) {:} 2$ & $23$ & $253$ & $33$ & $3$ & $3$ & - & \nsubN \\
    $\M_{23}$ & $\M_{22}$ & $1$ & $\PSL_3(4).2_2$ & $2$ & $2^4 {:} A7$ & $23$ & $253$ & $33$ & $3$ & $3$ & - & \nsubN \\
    $\M_{23}$ & $\M_{22}$ & $1$ & $\PSL_3(4).2_2$ & $2$ & $\PSL_3(4) {:} 2$ & $23$ & $253$ & $33$ & $3$ & $3$ & - & \nsubK \\
    $\M_{23}$ & $\M_{22}$ & $1$ & $2^4{:}\A_{7}$ & $3$ & $2^4 {:} A7$ & $23$ & $253$ & $33$ & $3$ & $3$ & - & \nsubK \\
    $\M_{23}$ & $\M_{22}$ & $1$ & $2^4{:}\A_{7}$ & $3$ & $\PSL_3(4) {:} 2$ & $23$ & $253$ & $33$ & $3$ & $3$ & - & \nsubN \\
    $\M_{23}$ & $\PSL_3(4).2_2$ & $2$ & $\M_{22}$ & $1$ & - & $253$ & $4554$ & $126$ & $7$ & $3$ & - & \nsubG \\
    $\M_{23}$ & $\PSL_3(4).2_2$ & $2$ & $\PSL_3(4).2_2$ & $2$ & - & $253$ & $4554$ & $126$ & $7$ & $3$ & - & \nsubG \\
    $\M_{23}$ & $\PSL_3(4).2_2$ & $2$ & $2^4{:}\A_{7}$ & $3$ & - & $253$ & $4554$ & $126$ & $7$ & $3$ & - & \nsubG \\
    $\M_{23}$ & $\PSL_3(4).2_2$ & $2$ & $\A_{8}$ & $4$ & - & $253$ & $4554$ & $126$ & $7$ & $3$ & - & \nsubG \\
    $\M_{23}$ & $\PSL_3(4).2_2$ & $2$ & $\M_{22}$ & $1$ & $2^3 {:} \PSL_3(2)$ & $253$ & $7590$ & $210$ & $7$ & $5$ & - & \norb \\
    $\M_{23}$ & $\PSL_3(4).2_2$ & $2$ & $\M_{22}$ & $1$ & $2^3 {:} \PSL_3(2)$ & $253$ & $7590$ & $210$ & $7$ & $5$ & $B_{1}$ & \ndes \\
    $\M_{23}$ & $\PSL_3(4).2_2$ & $2$ & $\PSL_3(4).2_2$ & $2$ & $2^3 {:} \PSL_3(2)$ & $253$ & $7590$ & $210$ & $7$ & $5$ & - & \nsubN \\
    $\M_{23}$ & $\PSL_3(4).2_2$ & $2$ & $\PSL_3(4).2_2$ & $2$ & $2^3 {:} \PSL_3(2)$ & $253$ & $7590$ & $210$ & $7$ & $5$ & - & \nsubN \\
    $\M_{23}$ & $\PSL_3(4).2_2$ & $2$ & $2^4{:}\A_{7}$ & $3$ & $2^3 {:} \PSL_3(2)$ & $253$ & $7590$ & $210$ & $7$ & $5$ & - & \norb \\
    $\M_{23}$ & $\PSL_3(4).2_2$ & $2$ & $2^4{:}\A_{7}$ & $3$ & $2^3 {:} \PSL_3(2)$ & $253$ & $7590$ & $210$ & $7$ & $5$ & $B_{2}$ & \ndes \\
    $\M_{23}$ & $\PSL_3(4).2_2$ & $2$ & $\A_{8}$ & $4$ & $2^3 {:} \PSL_3(2)$ & $253$ & $7590$ & $210$ & $7$ & $5$ & - & \norb \\
    $\M_{23}$ & $\PSL_3(4).2_2$ & $2$ & $\A_{8}$ & $4$ & $2^3 {:} \PSL_3(2)$ & $253$ & $7590$ & $210$ & $7$ & $5$ & $B_{3}$ & \ndes \\
    $\M_{23}$ & $\PSL_3(4).2_2$ & $2$ & $\M_{22}$ & $1$ & - & $253$ & $16698$ & $462$ & $7$ & $11$ & - & \nsubG \\
    $\M_{23}$ & $\PSL_3(4).2_2$ & $2$ & $\PSL_3(4).2_2$ & $2$ & - & $253$ & $16698$ & $462$ & $7$ & $11$ & - & \nsubG \\
    $\M_{23}$ & $\PSL_3(4).2_2$ & $2$ & $2^4{:}\A_{7}$ & $3$ & - & $253$ & $16698$ & $462$ & $7$ & $11$ & - & \nsubG \\
    $\M_{23}$ & $\PSL_3(4).2_2$ & $2$ & $\A_{8}$ & $4$ & - & $253$ & $16698$ & $462$ & $7$ & $11$ & - & \nsubG \\
    $\M_{23}$ & $\PSL_3(4).2_2$ & $2$ & $\M_{22}$ & $1$ & - & $253$ & $34914$ & $966$ & $7$ & $23$ & - & \nsubG \\
    $\M_{23}$ & $\PSL_3(4).2_2$ & $2$ & $\PSL_3(4).2_2$ & $2$ & - & $253$ & $34914$ & $966$ & $7$ & $23$ & - & \nsubG \\
    $\M_{23}$ & $\PSL_3(4).2_2$ & $2$ & $2^4{:}\A_{7}$ & $3$ & - & $253$ & $34914$ & $966$ & $7$ & $23$ & - & \nsubG \\
    $\M_{23}$ & $\PSL_3(4).2_2$ & $2$ & $\A_{8}$ & $4$ & - & $253$ & $34914$ & $966$ & $7$ & $23$ & - & \nsubG \\
    $\M_{23}$ & $2^4{:}\A_{7}$ & $3$ & $\M_{22}$ & $1$ & $2^4 {:} \GL_2(4)$ & $253$ & $3542$ & $140$ & $10$ & $5$ & - & \nsubN \\
    $\M_{23}$ & $2^4{:}\A_{7}$ & $3$ & $\PSL_3(4).2_2$ & $2$ & $2^4 {:} \GL_2(4)$ & $253$ & $3542$ & $140$ & $10$ & $5$ & - & \nsubN \\
    $\M_{23}$ & $2^4{:}\A_{7}$ & $3$ & $2^4{:}\A_{7}$ & $3$ & $2^4 {:} \GL_2(4)$ & $253$ & $3542$ & $140$ & $10$ & $5$ & - & \nsubN \\
    $\M_{23}$ & $2^4{:}\A_{7}$ & $3$ & $\A_{8}$ & $4$ & $2^4 {:} \GL_2(4)$ & $253$ & $3542$ & $140$ & $10$ & $5$ & - & \nsubN \\
    $\M_{23}$ & $2^4{:}\A_{7}$ & $3$ & $2^4{:}(3{\times}\A_{5}).2$ & $6$ & $2^4 {:} \GL_2(4)$ & $253$ & $3542$ & $140$ & $10$ & $
    5$ & - & \norb \\
    $\M_{23}$ & $\M_{11}$ & $5$ & $\M_{22}$ & $1$ & $2^4 {:} \S_5$ & $1288$ & $5313$ & $165$ & $40$ & $5$ & - & \norb \\
    $\M_{23}$ & $\M_{11}$ & $5$ & $\M_{22}$ & $1$ & $2^4 {:} \S_5$ & $1288$ & $5313$ & $165$ & $40$ & $5$ & - & \norb \\
    $\M_{23}$ & $\M_{11}$ & $5$ & $\PSL_3(4).2_2$ & $2$ & $2^4 {:} \S_5$ & $1288$ & $5313$ & $165$ & $40$ & $5$ & - & \norb \\
    $\M_{23}$ & $\M_{11}$ & $5$ & $\PSL_3(4).2_2$ & $2$ & $2^4 {:} \S_5$ & $1288$ & $5313$ & $165$ & $40$ & $5$ & - & \norb \\
    $\M_{23}$ & $\M_{11}$ & $5$ & $2^4{:}\A_{7}$ & $3$ & $2^4 {:} \S_5$ & $1288$ & $5313$ & $165$ & $40$ & $5$ & - & \norb \\
    $\M_{23}$ & $\M_{11}$ & $5$ & $2^4{:}\A_{7}$ & $3$ & $2^4 {:} \S_5$ & $1288$ & $5313$ & $165$ & $40$ & $5$ & - & \norb \\
    $\M_{23}$ & $\M_{11}$ & $5$ & $2^4{:}(3{\times}\A_{5}).2$ & $6$ & $2^4 {:} \S_5$ & $1288$ & $5313$ & $165$ & $40$ & $5$ & - & \norb \\
    $\M_{23}$ & $\M_{11}$ & $5$ & $2^4{:}(3{\times}\A_{5}).2$ & $6$ & $2^4 {:} \S_5$ & $1288$ & $5313$ & $165$ & $40$ & $5$ & - & \norb \\
    $\M_{24}$ & $\M_{12}.2$ & $4$ & $\M_{22}.2$ & $2$ & - & $1288$ & $27324$ & $297$ & $14$ & $3$ & - & \nsubG \\
    $\M_{24}$ & $\M_{12}.2$ & $4$ & $2^4{:}\A_{8}$ & $3$ & - & $1288$ & $27324$ & $297$ & $14$ & $3$ & - & \nsubG \\
    $\M_{24}$ & $\M_{12}.2$ & $4$ & $2^4{:}\A_{8}$ & $3$ & - & $1288$ & $5313$ & $165$ & $40$ & $5$ & - & \nsubG \\
    $\M_{24}$ & $\M_{12}.2$ & $4$ & $2^6{:}3.\S_{6}$ & $5$ & - & $1288$ & $5313$ & $165$ & $40$ & $5$ & - & \nsubG \\
    $\M_{24}$ & $\M_{12}.2$ & $4$ & $\M_{22}.2$ & $2$ & - & $1288$ & $45540$ & $495$ & $14$ & $5$ & - & \nsubG \\
    $\M_{24}$ & $\M_{12}.2$ & $4$ & $2^4{:}\A_{8}$ & $3$ & - & $1288$ & $45540$ & $495$ & $14$ & $5$ & - & \nsubG \\
    $\M_{24}$ & $\M_{12}.2$ & $4$ & $2^6{:}(\PSL_{3}(2){\times}\S_{3})$ & $7$ & - & $1288$ & $45540$ & $495$ & $14$ & $5$ & - & \nsubG \\
    $\M_{24}$ & $\M_{12}.2$ & $4$ & $\M_{22}.2$ & $2$ & $2^6 {:} \A_5$ & $1288$ & $63756$ & $693$ & $14$ & $7$ & - & \nsubN \\
    $\M_{24}$ & $\M_{12}.2$ & $4$ & $\M_{22}.2$ & $2$ & $(2^4 {:} \A_5) {:} 2^2$ & $1288$ & $63756$ & $693$ & $14$ & $
    7$ & - & \nsubN \\
    $\M_{24}$ & $\M_{12}.2$ & $4$ & $\M_{22}.2$ & $2$ & $(2^4 {:} \A_5) {:} 4$ & $1288$ & $63756$ & $693$ & $14$ & $7$ & - & \nsubN \\
    $\M_{24}$ & $\M_{12}.2$ & $4$ & $\M_{22}.2$ & $2$ & $(2^4 {:} \A_5) {:} 2^2$ & $1288$ & $63756$ & $693$ & $14$ & $
    7$ & - & \nsubK \\
    $\M_{24}$ & $\M_{12}.2$ & $4$ & $\M_{22}.2$ & $2$ & $(2^6 {:} 15) {:} 4$ & $1288$ & $63756$ & $693$ & $14$ & $
    7$ & - & \nsubN \\
    $\M_{24}$ & $\M_{12}.2$ & $4$ & $2^4{:}\A_{8}$ & $3$ & $2^6 {:} \A_5$ & $1288$ & $63756$ & $693$ & $14$ & $7$ & - & \nsubN \\
    $\M_{24}$ & $\M_{12}.2$ & $4$ & $2^4{:}\A_{8}$ & $3$ & $(2^4 {:} \A_5) {:} 2^2$ & $1288$ & $63756$ & $693$ & $14$ & $
    7$ & - & \nsubN \\
    $\M_{24}$ & $\M_{12}.2$ & $4$ & $2^4{:}\A_{8}$ & $3$ & $(2^4 {:} \A_5) {:} 4$ & $1288$ & $63756$ & $693$ & $14$ & $7$ & - & \nsubN \\
    $\M_{24}$ & $\M_{12}.2$ & $4$ & $2^4{:}\A_{8}$ & $3$ & $(2^4 {:} \A_5) {:} 2^2$ & $1288$ & $63756$ & $693$ & $14$ & $
    7$ & - & \nsubN \\
    $\M_{24}$ & $\M_{12}.2$ & $4$ & $2^4{:}\A_{8}$ & $3$ & $(2^6 {:} 15) {:} 4$ & $1288$ & $63756$ & $693$ & $14$ & $
    7$ & - & \nsubN \\
    $\M_{24}$ & $\M_{12}.2$ & $4$ & $2^6{:}3.\S_{6}$ & $5$ & $2^6 {:} \A_5$ & $1288$ & $63756$ & $693$ & $14$ & $7$ & - & \nsubK \\
    $\M_{24}$ & $\M_{12}.2$ & $4$ & $2^6{:}3.\S_{6}$ & $5$ & $(2^4 {:} \A_5) {:} 2^2$ & $1288$ & $63756$ & $693$ & $14$ & $
    7$ & - & \nsubK \\
    $\M_{24}$ & $\M_{12}.2$ & $4$ & $2^6{:}3.\S_{6}$ & $5$ & $(2^4 {:} \A_5) {:} 4$ & $1288$ & $63756$ & $693$ & $14$ & $7$ & - & \nsubK \\
    $\M_{24}$ & $\M_{12}.2$ & $4$ & $2^6{:}3.\S_{6}$ & $5$ & $(2^4 {:} \A_5) {:} 2^2$ & $1288$ & $63756$ & $693$ & $14$ & $7$ & - & \nsubK \\
    $\M_{24}$ & $\M_{12}.2$ & $4$ & $2^6{:}3.\S_{6}$ & $5$ & $(2^6 {:} 15) {:} 4$ & $1288$ & $63756$ & $693$ & $
    14$ & $7$ & - & \nsubK \\
    $\M_{24}$ & $\M_{12}.2$ & $4$ & $\M_{22}.2$ & $2$ & - & $1288$ & $209484$ & $2277$ & $14$ & $23$ & - & \nsubG \\
    $\M_{24}$ & $\M_{12}.2$ & $4$ & $2^4{:}\A_{8}$ & $3$ & - & $1288$ & $209484$ & $2277$ & $14$ & $23$ & - & \nsubG \\
    $\HS$ & $\PSU_3(5).2$ & $2$ & $2^4.\S_{6}$ & $6$ & $(2^4 {:} \A_6) {:} 2$ & $176$ & $3850$ & $175$ & $8$ & $7$ & - & \nsubK \\
    $\HS$ & $\PSU_3(5).2$ & $3$ & $2^4.\S_{6}$ & $6$ & $(2^4 {:} \A_6) {:} 2$ & $176$ & $3850$ & $175$ & $8$ & $7$ & - & \nsubK \\
    $\HS{:}2$ & $\M_{22}.2$ & $2$ & $\HS$ & $1$ & - & $100$ & $330$ & $33$ & $10$ & $3$ & - & \nsubG \\
    $\HS{:}2$ & $\M_{22}.2$ & $2$ & $\HS$ & $1$ & - & $100$ & $550$ & $55$ & $10$ & $5$ & - & \nsubG \\
    $\HS{:}2$ & $\M_{22}.2$ & $2$ & $\HS$ & $1$ & - & $100$ & $770$ & $77$ & $10$ & $7$ & - & \nsubG \\
    $\J_{2}$ & $\PSU_3(3)$ & $1$ & $2^{2+4}.3{\times}\S_3$ & $4$ & 
    $(([2^6] {:} 3) {:} 2) {:} 3$
    & $100$ & $525$ & $
    63$ & $12$ & $7$ & $B_{4}$ & \ndes \\
    $\J_2{:}2$ & $\PSU_3(3).2$ & $2$ & $2^{2+4}{:}(\S_3{\times}\S_3)$ & $5$ & 
    $(([2^6] {:} 3) {:} 2) {:} 3$
    & $100$ & $
    525$ & $63$ & $12$ & $7$ & $B_{5}$ & \ndes \\
    \hline
    \multicolumn{13}{l}{The $K$-orbits $B_{1}$, $B_{2}$ and $B_{3}${:}}\\ 
    \multicolumn{13}{l}{$B_1=\{50, 64, 101, 142, 187, 202, 242 \}$}\\
    \multicolumn{13}{l}{$B_2=\{4, 38, 63, 134, 162, 200, 215\}$}\\
    \multicolumn{13}{l}{$B_3=\{ 5, 34, 69, 149, 201, 221, 243\}$}\\
    \multicolumn{13}{l}{$B_4=\{6, 12, 23, 24, 26, 34, 39, 42, 68, 70, 78, 86\}$}\\
    \multicolumn{13}{l}{$B_5=\{10, 30, 36, 44, 46, 49, 64, 67, 68, 75, 86, 93\}$}\\
\end{longtable}
\normalsize




\end{document}